\documentclass[12pt]{amsart}
\usepackage{amsmath}
\usepackage{mathtools}

\DeclarePairedDelimiter{\norm}{\lVert}{\rVert}

\usepackage{amsfonts}
\usepackage{amsthm}
\usepackage{autobreak}
\usepackage[english]{babel}
\usepackage{bold-extra}
\usepackage{hyperref}
\usepackage{comment}
\usepackage{mdwlist}
\theoremstyle{definition}
\newtheorem{definition}{Definition}
\newtheorem*{problem}{Problem}
\newtheorem*{problemQSAIP}{Problem QSAIP (Quantitative Simultaneous Approximation by Incomplete Polynomials)}
\newtheorem{remark}[definition]{Remark}
\theoremstyle{plain}
\newtheorem{prop}[definition]{Proposition}
\newtheorem{cor}[definition]{Corollary}
\newtheorem{conjecture}{Conjecture}
\newtheorem{lemma}[definition]{Lemma}
\newtheorem{theorem}[definition]{Theorem}

\newtheorem*{theoremA}{Theorem A}
\newtheorem*{theoremB}{Theorem B}
\newtheorem*{theoremC}{Theorem C}
\newtheorem*{theoremD}{Theorem D}
\newtheorem*{theoremE}{Theorem E}
\newtheorem*{theoremBW}{Bernstein-Walsh Theorem}

\newtheorem{question}[definition]{Question}

\newtheorem*{questionA}{Question A}
\newtheorem*{questionB}{Question B}
\newtheorem{example}[definition]{Example}

\numberwithin{definition}{section}
\allowdisplaybreaks

\def\D{{\mathbb{D}}}

\def\R{{\mathbb{R}}}
\def\C{{\mathbb{C}}}
\def\Q{{\mathbb{Q}}}
\def\N{{\mathbb{N}}}

\def\T{{\mathbb{T}}}
\def\UU{{\mathcal{U}}}

\def\LL{{\mathcal{L}}}

\def\EE{{\mathcal{E}}}

\def\DD{{\mathcal{D}}}

\def\PP{{\mathcal{P}}}

\def\MM{{\mathcal{M}}}

\def\dlow{\underline{d}}
\def\dup{\overline{d}}
\newcommand{\dlowg}[1]{\underline{d}_{#1}}

\numberwithin{equation}{section} 
\numberwithin{figure}{section} 
\numberwithin{table}{section} 

\newcommand{\indicatrice}[1]{\mathds{1}_{#1}}

\newcommand{\vp}{\varphi}

\usepackage{multicol}
\usepackage[T1]{fontenc}
\usepackage[utf8]{inputenc}
\usepackage[a4paper]{geometry}
\usepackage{amssymb,amsmath}
\usepackage{version}
\usepackage{mathtools} 
\usepackage{enumerate}
\usepackage{dsfont}
\usepackage{mathdots}
\makeatletter

\usepackage{latexsym}
\usepackage{color}

\begin{document}
	
	\title[Incomplete polynomial approximation and frequently universal series]{Quantitative incomplete polynomial approximation and frequently universal Taylor series}
	
	\author{St\'{e}phane Charpentier, Konstantinos Maronikolakis}
	
	\address{St\'ephane Charpentier, Aix-Marseille Univ, CNRS, I2M, Marseille, France}
	\email{stephane.charpentier.1@univ-amu.fr}
	
	\address{Konstantinos Maronikolakis, Department of Mathematics, Bilkent University, 06800 Ankara, Turkey}
	\email{conmaron@gmail.com}
	
	\thanks{The research was conducted during the stay of the second author in the Institut de Mathématiques de Marseille, and was funded by the Embassy of France in Ireland through the 2024 France Excellence Research Residency grant. Konstantinos Maronikolakis also acknowledges financial support from the Irish Research Council through the grant [GOIPG/2020/1562]. The first author is partially supported by the grant ANR-24-CE40-0892-01 of the French National Research Agency
		ANR (project ComOp).}
	
	
	\keywords{Incomplete polynomial, Weighted polynomial approximation, universal Taylor series, frequently universal operators}
	\subjclass[2020]{30K15, 41A10, 30E10, 41A25}

	\begin{abstract}Let $(\tau_n)_n$ be a sequence of real numbers in $(1,+\infty)$. Using potential theoretic methods, we prove quantitative results  -- Bernstein-Walsh type theorems -- about uniform approximation by polynomials of the form $\sum_{k=\lfloor \frac{n}{\tau_n} \rfloor}^na_k z^k$, on the union of two disjoint compact sets, one containing 0 and the other not. Moreover, we reveal the interplay between the compact sets and the asymptotic behaviour of the sequence $(\tau_n)_n$. As applications of our results, we prove the existence of frequently universal Taylor series, with respect to the natural and the logarithmic densities, providing solutions to two problems posed by Mouze and Munnier.
	\end{abstract}
	
	\maketitle
	
	\section{Introduction}
	
	Approximation by incomplete polynomials constitutes a natural and intriguing variation of classical polynomial approximation, that goes back to Lorentz \cite{LorentzBook1977,LorentzBook1980}. 	Let $(\tau_n)_n$ be a sequence of real numbers with $\tau_n >1$, $n\in \N$. A complex polynomial $P$ is said to be \textit{incomplete with respect to $(\tau_n)_n$} if there exist $n\in \N$ and $a_{\lfloor \frac{n}{\tau_n}\rfloor},\ldots,a_n \in \C$ such that
	\begin{equation}\label{eq-def-inc-poly}
		P(z)=\sum_{k=\lfloor \frac{n}{\tau_n} \rfloor}^na_kz^k,
	\end{equation}
	where $\lfloor x \rfloor$ denotes the integer part of $x$. Over the past decades, incomplete polynomial approximation has emerged as a subject in its own right, driven not only by applications, notably in numerical analysis, but also by its intrinsic mathematical richness, see for example \cite{Saff1983,SaffUllmanVarga1980}. In particular, it has been successfully developed from the point of view of weighted polynomial approximation, with tools from potential theory \cite{MhaskarSaff1985,SaffTotikBookLogpotextfields,PritskerVarga1998}. In this paper, we investigate further the topic, motivated by a novel and purely theoretical application to the study of frequently universal Taylor series.
	
	Frequent hypercyclicity, introduced in \cite{BayartGrivaux2006}, is one of the central concepts in linear dynamics, see \cite{Menet2017,BayartRusza2015,GrivauxMatheron2014}. Given an infinite dimensional Fréchet space $X$, a bounded linear operator $T:X\to X$ is hypercyclic provided there exists $x\in X$, such that the orbit of $x$ under $T$ is dense in $X$, namely the sets $\left\{n\in\N:T^n x\in U\right\}$ are infinite for all non-empty open sets $U$ of $X$. Frequent hypercyclicity quantifies the size of these sets, in terms of density for subsets of the set $\N$ of all non-negative integers. Let us recall that the natural lower density of a set $E\subset \N$, that we shall denote by $\dlow(E)$, is defined by
	\begin{equation}\label{eq-def-dlow}
		\dlow(E)=\liminf_{n\to \infty}\frac{|E\cap \{0,\ldots,n\}|}{n},
	\end{equation}
	where $|E|$ is the cardinality of $E$. Thus, an operator is frequently hypercyclic whenever there exists $x\in X$ such that for any non-empty open set $U$ of $X$,
	\[
	\dlow\left(\left\{n\in\N:T^n x\in U\right\}\right)>0 .
	\]
	We refer the reader to the textbooks \cite{BayartMatheron2009,GrosseErdmannPerisManguillot2011} for a rich introduction to linear dynamics.
	
	Although frequent hypercyclicity has been the subject of systematic study for a relatively short time, research had already began around the frequent hypercyclicity properties of certain Dirichlet series, for example the Riemann $\zeta$ function, more than 40 years ago. Let $h>0$ and let $T_h$ be the vertical translation by $h$, acting from the Fréchet space of all functions $f$ holomorphic on $\{s\in \C:\,1/2< \Im(s)<1\}$ into itself, defined by $T_h(f)(\cdot)=f(\cdot + ih)$. Improving the celebrated result of Voronin on the hypercyclic properties of the $\zeta$ function for $T_h$ \cite{Voronin1975}, Reich showed in \cite{Reich1980} that, given any $\varepsilon>0$, any compact set $K \subset  \{s\in \C:\,1/2< \Im(s)<1\}$ and any zero-free function $\varphi\in A(K)$, we have
	\[
	\dlow\left(\left\{n\in \N:\, \max_{z\in K}\left|T_h^n(f)(z)-\varphi(z)\right|<\varepsilon\right\}\right)>0.
	\]
	Recent developments in this direction can be found in \cite{BayartKouroupis2024,Sourmelidis2023}. Nowadays,  the frequent hypercyclicity criterion proves to be useful for checking that an explicit operator is frequently hypercyclic, and the frequent hypercyclicity of operators from standard classes of operators (weighted shifts, composition operators,...) is well-studied, and rather well understood for weighted shifts \cite{BayartRusza2015,BonillaGrosseErdmann2007,CharpentierGrosseErdmannMenet2021}.
	
	Linear dynamics is a subfield of the broader theory of universality in operator theory, that extends the above notions to sequences of operators that are not necessarily iterates of a single one. Given two Fréchet spaces, $X$ and $Y$, and a sequence $(T_n)_n$ of bounded linear operators from $X$ to $Y$, we say that $(T_n)_n$ is universal if there exists $x\in X$ such that the generalized orbit $\left\{T_n x:\, n\in \N\right\}$ of $x$ under $(T_n)_n$ visits infinitely many times any non-empty open sets of $Y$. We refer to \cite{GrosseErdmann1999,BayartGrosseErdmanNestoridisPapadimitropoulos2008} for background on the theory of universality. Likewise, the vector $x$ is frequently universal for $(T_n)_n$ if the sets $\left\{n\in\N:T_n x\in U\right\}$ have positive lower density, for all non-empty open sets $U$ of $Y$. Criteria for frequent universality were given in \cite{BonillaGrosseErdmann2007,BonillaGrosseErdmann2009}. Amongst the numerous concrete examples of universal sequences of operators, universal Taylor series is one of them that has been thoroughly studied for 30 years.
	
	From now on, if $\Omega$ is a domain in $\C$, we shall denote by $H(\Omega)$ the Fréchet space of all holomorphic functions of $\Omega$, endowed with the topology of locally uniform convergence. We will denote the unit disc by $\D$ and the unit circle by $\T$. For a fixed compact set $K\subset \C\setminus \D$, we will denote by $C(K)$ the Banach space of all continuous functions on $K$, endowed with the supremum norm $\Vert \cdot \Vert _K$. The notation $A(K)$ will stand for the subspace of $C(K)$ consisting of the functions that are holomorphic in the interior of $K$. Then, for $n\in\N$, let $S_n^{(K)}$ denote the $n$-th partial sum operator, defined by
	\[
	S_n^{(K)}:\left\{\begin{array}{lll}
		H(\D) & \to & A(K)\\
		f=\sum _{k=0}^\infty a_kz^k & \mapsto & \sum_{k=0}^n a_kz^k.
	\end{array}\right.
	\]
	The study of the universal properties of $(S_n^{(K)})_n$ was initiated by Chui and Parnes in \cite{ChuiParnes1971} and led to the quite active theory of universal Taylor series.
	\begin{definition}[$K$-Universal Taylor series]Let $K\subset \C\setminus \D$ be compact. A function $f$ in $H(\D)$ is called a $K$-universal Taylor series if $f$ is universal for $(S_n^{(K)})_n$. We denote by $\UU(\D,K)$ the set of all such functions.
	\end{definition}
	Denoting by $\MM$ the set of all compact subsets of $\C\setminus \D$ with connected complement, a Baire argument was used by Nestoridis \cite{Nestoridis1996} to show that the set
	\[
	\UU(\D):=\bigcap_{K\in \MM}\UU(\D,K)
	\]
	is a dense $G_{\delta}$-subset of $H(\D)$. Various properties of universal Taylor series have been intensively studied, we refer to \cite{MelasNestoridisPapadoperakis1997,MelasNestoridis2001,MullerVlachouYavrian2006,Gardiner2014,GardinerKhavinson2014} for a small sample of indicative results. In 2013, Papachristodoulos initiated a study of frequently universal Taylor series \cite{Papachristodoulos2013}:
	\begin{definition}[Frequently universal Taylor series]\label{definition-intro-FUTS}A function $f$ in $H(\D)$ is called a \textit{frequently universal Taylor series} if, for any compact set $K\in \MM$ and any non-empty open set $U$ in $A(K)$,
		\[
		\dlow\left(\left\{n\in \N: \, S_n^{(K)}(f) \in U\right\}\right)>0.
		\]
		The set of all frequently universal Taylor series is denoted by $FU(\D)$.
	\end{definition}
	In \cite{Papachristodoulos2013}, the author shows that $FU(\D)$ is meagre in $H(\D)$. Later, Mouze and Munnier proved that the set $FU(\D)$ is actually empty, and extend this to more general weighted densities, that are larger than the natural one, but smaller than the logarithmic density \cite{MouzeMunnier2017,MouzeMunnier2016} (the precise definition of the logarithmic density is given in Section \ref{prelim-1-1}). Thus they ask the following question:
	\begin{questionA}[Problem 2, page 944 in \cite{MouzeMunnier2016}]\label{equt-1-intro}Do there exist functions in $H(\D)$ that are frequently universal with respect to the \textit{logarithmic density}?
	\end{questionA}
	
	The reader who is familiar with universality or linear dynamics may have observed that functions in $\UU(\D)$ are ``common'' universal functions for all sequences of operators $(S_n^{(K)})_n$, $K\in \MM$. Now, it turns out that common frequent hypercyclicity has been proven to be a very restrictive notion for most non-trivial families of operators, see for example \cite{CharpentierErnstMestiriMouze2022}. Thus it is natural to wonder whether it may be a reason that explains why $FU(\D)=\emptyset$. The proofs in \cite{MouzeMunnier2016,MouzeMunnier2017} confirm this, as they show that in order to arrive at the latter result, it is very useful to exploit the universal approximation properties of functions in $\UU(\D)$ on compact sets with non-empty interior, and diameter going to $\infty$. We mention that the non-existence of frequently universal Taylor series can also be directly derived from the useful fact that every function in $\UU(\D)$ possesses Ostrowski-gaps \cite{Charpentier2019}, a property that is also proved using approximation on ``large'' compact sets \cite{GehlenLuhMuller2000}.
	
	\medskip
	
	This leads to weakening Definition \ref{definition-intro-FUTS} in the following way:
	\begin{definition}[$K$-Frequently universal Taylor series]\label{definition-intro-K-FUTS}Let $K\subset \C\setminus \D$. A function $f$ in $H(\D)$ is a \textit{$K$-frequently universal Taylor series} if, for any non-empty open set $U$ in $A(K)$,
		\[
		\dlow\left(\left\{n\in \N: \, S_n^{(K)}(f) \in U\right\}\right)>0.
		\]
		The set of all $K$-frequently universal Taylor series is denoted by $FU(\D,K)$.
	\end{definition}
	
	Mouze and Munnier raised the following problem:
	\begin{questionB}[Problem 3, page 945 in \cite{MouzeMunnier2016}]\label{quest-2-intro}Let $K\subset \C \setminus \D$. Do there exist $K$-frequently universal Taylor series?
	\end{questionB}

	Achieving the initial motivation for this work, we give in this paper a positive answer to Question A and a partial positive answer to Question B:
	\begin{theoremA}\label{thma}There exist a dense set of functions $f$ in $H(\D)$ that are $K$-frequently universal Taylor series for any $K\in \MM$, with respect to the logarithmic density.
	\end{theoremA}
	
	\begin{theoremB}\label{thmb}For any compact set $K\subset \C\setminus \overline{\D}$, with connected complement, the set $FU(\D,K)$ is non-empty.
	\end{theoremB}
	
	Theorem \hyperref[thmb]{B} is only a partial answer to Question B since it does not cover the case where $K$ intersects the unit circle. So far we are not able to treat this case.
	
	\medskip
	
	Let us comment on the ideas of the proofs, as they lead to the topic of approximation by incomplete polynomials. The usual way to build a universal Taylor series $f$ in $H(\D)$ is to write $f$ as a series $\sum_n P_n$ of polynomials $P_n$, that are inductively chosen with good \textit{simultaneous} approximation properties. More precisely, fix a compact set $K \in \MM$ and a countable basis $(U_n)_n$ of the topology of $A(K)$. The $P_n$'s are chosen of the form $\sum_{v_n}^{d_n}a_kz^k$, with coefficients that do not overlap (that is, $v_n>d_{n-1}$) in such a way that $P_n$ is small in $H(\D)$ and $P_0 +\ldots +P_n$ belongs to $U_n$. Then, for any $n \in \N$, $S_{d_n}^{(K)}(f)=P_0+\ldots + P_n \in U_n$. The existence of the polynomials $P_n$ follows from Mergelyan's Theorem which states that any function in $A(K)$ can be uniformly approximated on $K$ by polynomials. A suitable application of Mergelyan's Theorem can also give us the condition $v_n>d_{n-1}$. Now, in order to ensure that the partial sums $S_k^{(K)}(f)$, $k\in \N$, visit frequently every non-empty open set $U_n$,  the idea is to assign to each $U_n$ a subset $E_n$ of $\N$, with positive density, such that, for any $k\in E_n$, $S_k^{(K)}(f) $ belongs to $U_n$. In order to accomplish this, it is basically enough to be able, given any pair $(E_n,E_m)$ with $n\neq m$ and any $(k,l)\in E_n\times  E_m$ with $k>l$, to find a polynomial $P$ of the form $\sum_{i=l+1}^ka_iz^i$ that lies in $U_n$ and is small in $H(\D)$. This is where quantitative approximation by incomplete polynomials comes into play.
	
	\medskip
	
	Thus the following independent and interesting problem on complex polynomial approximation arises:
	
	\begin{problemQSAIP}\label{qsaip}Given two disjoint compact sets $K\subset \C\setminus \{0\}$ and $L \subset \C$, with connected complement,  find a description of the sequences $(\tau_n)_n$ of real numbers in $(1,+\infty)$ such that any function in $A(K)$ can be uniformly approximated on $K$ by polynomials of the form $\sum_{k=\lfloor \frac{n}{\tau_n}\rfloor }^n a_kz^k$, that simultaneously approximate $0$ uniformly on $L$. Moreover, find a quantitative approximation result of this kind.
	\end{problemQSAIP}
	
	It is worth mentioning that, in Problem \hyperref[qsaip]{QSAIP}, the compact set $L$ may contain $0$, which is the most interesting case. The less restrictive problem of approximation by incomplete polynomials only on $K$ (not considering simultaneous approximation around $0$) has attracted much attention, especially when $K$ is an interval of the real line \cite{LorentzBook1977,LorentzBook1980,SaffVarga1978,SaffVarga1979,Golitschek1980,LorentzolitschekMakovoz1996}. In the complex plane, this has been successfully tackled from the point of view of weighted polynomial approximation, using advanced tools from potential theory \cite{PritskerVarga1998,Kuijlaars1996,SaffTotikBookLogpotextfields,TotikVarju2007}; see Section \ref{wpa} for the presentation of a very nice result of Pritsker and Varga \cite{PritskerVarga1998}.
	
	The first part of this paper is thus devoted to investigating solutions to Problem \hyperref[qsaip]{QSAIP}. This part can be read independently from the one dealing with frequently universal Taylor series. Thanks to a certain refinement of the classical Bernstein-Walsh Theorem \cite{Ransford1995}, and to variants of classical results from potential theory, we are able to prove the following three results. From now on, and without possible confusion, we shall simply write $\sum_{k=\frac{n}{\tau_n}}^na_kz^k$ instead of $\sum_{k=\lfloor \frac{n}{\tau_n} \rfloor}^na_kz^k$, where $\lfloor x \rfloor =\max\{n\in \N: n\leq x\}$, $x\in [0,+\infty)$. With the notations used in the statement of Problem \hyperref[qsaip]{QSAIP}, we will first consider the case where the sequence $(\tau_n)_n$ is constant:
	
	\begin{theoremC}\label{thmc}Let $K\subset \C\setminus \{0\}$ and $L\subset \C$ be two disjoint compact sets with connected complement. Let also $U$ be a bounded open set containing $K$. Then there exist $\tau \geq 1$, $\theta\in (0,1)$ and $N\in \N$ such that for any $\vp \in A(\overline{U})$ and any $n \geq N$, there exists a polynomial $P_n$ of the form $\sum_{k= n/\tau}^{n} a_k z^k$ such that
		\[
		\Vert P_n \Vert _L \leq \Vert \vp \Vert _{\overline{U}}\theta^{n}\quad \text{and}\quad \Vert P_n - \vp \Vert _K \leq \Vert \vp \Vert _{\overline{U}}\theta^{n}.
		\]
	\end{theoremC}
	
	It will appear from the proof how the constant $\tau$ depends on $K$ and $L$. The next theorem corresponds to the case where $\tau_n \to +\infty$.
	
	\begin{theoremD}\label{thmd}Let $(\tau_n)_n$ be a sequence of positive real numbers with $\tau _n \to +\infty$. For any disjoint compact sets $K\subset \C\setminus \{0\}$ and $L\subset \C$ , with connected complement, and for any bounded open set $U$ containing $K$, there exist $\theta\in (0,1)$ and $N\in \N$ such that, for any $\vp \in A(\overline{U})$ and any $n \geq N$, there exists a polynomial $P_n$ of the form $\sum_{k=n/\tau_n}^{n}a_kz^k$ such that
		\[
		\Vert P_n \Vert _L \leq \Vert \vp \Vert _{\overline{U}}\theta^{n}\quad \text{and}\quad \Vert P_n - \vp\Vert _K \leq \Vert \vp \Vert _{\overline{U}}\theta^{n}.
		\]
	\end{theoremD}
	
	In our third result, we are interested in the case where $\tau_n \to 1$. In this situation, we have to assume that $L=\emptyset$ (see Section \ref{pot} for the definition of a polar set).
	
	\begin{theoremE}\label{thme}Let $K\subset \C\setminus \{0\}$ be a polar compact set and let $U$ be a bounded open set containing $K$. There exist a sequence $(\tau_n)_n$ in $(1,\infty)$ with $\tau_n\to1$, a sequence $(\eta_n)_n$ of positive real numbers with $\eta_n\to 0$, and $N\in \N$, such that for any $\vp\in A(\overline{U})$ and any $n\geq N$, there exists a polynomial $P_n$ of the form $\sum_{k= n/\tau_n}^{n}a_kz^k$ such that
		\[
		\Vert P_n-\vp\Vert_K\leq \Vert \vp\Vert _{\overline{U}}\eta_n^{n}.
		\]
	\end{theoremE}
	
	We mention that Theorem \hyperref[thmd]{D} is a slightly more precise reformulation of a result by Costakis and Tsirivas \cite{CostakisTsirivas2014}. Theorem \hyperref[thmc]{C} with $L=\emptyset$ was obtained in \cite{CharpentierLevenbergWielonsky2024}, using the approach of Pritsker and Varga in \cite{PritskerVarga1998}. We will see that, to some extent, these theorems are sharp. For example, we will see that if a compact set supports approximation by polynomials of the form $\sum_{k=n/\tau_n}^na_kz^k$, with $\tau_n \to 1$, then $K$ must have empty interior.
	
	As an application, Theorem \hyperref[thmc]{C} will be used to prove Theorem \hyperref[thmb]{B}, and Theorem \hyperref[thmd]{D} to prove Theorem \hyperref[thma]{A}. We will discuss how Theorem \hyperref[thme]{E} could be a first step towards building frequently universal Taylor series on polar sets, with respect to weighted densities that are strictly smaller than the natural one. Finally, let us say that quantitative incomplete polynomial approximation has been used earlier in the construction of universal Taylor series. Indeed, a result similar to Theorem \hyperref[thmd]{D} was used by Costakis and Tsirivas to build so-called doubly universal Taylor series \cite{CostakisTsirivas2014}, see also \cite{Mouze2018,Charpentier2019,ChatzigiannakidouVlachou2016}. In \cite{CharpentierMaronikolakisII2025}, we apply the above-mentioned result of Pritsker and Varga \cite{PritskerVarga1998}, in order to tackle some other open problems about universal Taylor series.
	
	\medskip
	
	The paper is organized as follows. Section \ref{sec-2}, devoted to approximation by incomplete polynomials, is divided into 4 subsections: in the first one, we recall the basic ingredients and the main tools from potential theory that will be useful to us; the second one is dedicated to introducing the topic of approximation by incomplete polynomials, and to proving Theorems \hyperref[thmc]{C} and \hyperref[thmd]{D}. In the third subsection, we prove Theorem \hyperref[thme]{E}. We present some open problems related to the topic in the fourth subsection. Section \ref{sec-FUTS} deals with applying the results of Section \ref{sec-2} to frequently universal Taylor series. The notion of weighted density is first introduced, then Theorems \hyperref[thma]{A} and \hyperref[thmb]{B} are proven. In a final subsection, problems for further research on frequently universal Taylor series are discussed.

	\section{Approximation by incomplete polynomials}\label{sec-2}
	\subsection{Tools from potential theory}\label{pot}
	Potential theory plays a crucial role in the development of approximation theory. Thus, in this section, we give some preliminary potential theoretic notions that will be used throughout the paper. We follow the textbook of Ransford \cite{Ransford1995} for the notation and definitions.
	
	First, let us state the Harnack inequalities for a general domain, which is a crucial property of positive harmonic functions:
	\begin{theorem}[Harnack inequalities]
		Let $D$ be a domain in $\C$. Then, for any $z,w\in D$, there exists a positive number $d$ such that, for any positive harmonic function $h$ on $D$ we have
		$$\frac{1}{d}\cdot h(w)\leq h(z)\leq d\cdot h(w).$$
	\end{theorem}
	For $D,z,w$ as above, we call the \emph{Harnack distance} of $z$ and $w$ the smallest number $d$ as in the previous theorem. We denote the Harnack distance of $z$ and $w$ by $d_D(z,w)$.
	
	Next, we give the standard definitions of classical (logarithmic) potential theory. Let $\mu$ be a finite Borel measure on $\C$ with compact support. Then, the function $p_{\mu}:\C\to[-\infty,\infty)$ with
	$$p_{\mu}(z)=\int\log|z-w|d\mu(w),z\in\C$$
	is called the \emph{(logarithmic) potential} of $\mu$. Potentials are subharmonic functions on $\C$.
	\begin{definition}[Energy]
		The \emph{energy} of a finite Borel measure $\mu$ on $\C$ is defined as the quantity
		$$I(\mu):=\int p_{\mu}(z)d\mu(z)=\int\int\log|z-w|d\mu(w)d\mu(z).$$
	\end{definition} 
	Further, let us fix a compact set $K$ in $\C$. A probability measure $\nu$ supported on $K$ is called an \emph{equilibrium measure} for $K$, if $I(\nu)=\sup_{\mu} I(\mu)$ where the supremum is taken over all probability measures $\mu$ supported on $K$. It can be shown that an equilibrium measure for $K$ always exists and is supported on $\partial K$.
	\begin{definition}[Capacity]
		The \emph{(logarithmic) capacity} of a compact set $K$ in $\C$ is defined as $c(K):=e^{I(\nu)}$, where $\nu$ is an equilibrium measure for $K$.
	\end{definition}
	We note that, for $K$ and $\nu$ as above, the case $I(\nu)=-\infty$ is possible, which is equivalent to saying that any probability measure supported on $K$ has energy equal to $-\infty$. We call these sets \emph{polar}. Clearly, this is equivalent to $c(K)=0$.
	
	The capacity of a compact set can also be understood through the notion of the transfinite diameter.
	\begin{definition}
		Let $K$ be a compact set in $\C$ and $n\geq2$. The \emph{$n$-th diameter} of $K$ is defined as
		$$\delta_n(K):=\sup\left\{\prod_{1\leq j<k\leq n}|w_j-w_k|^{\frac{2}{n(n-1)}}:w_1,\dots,w_n\in K\right\}.$$
		Also, we call an $n$-tuple $(w_1,\dots,w_n)\in K^n$ for which the supremum is attained a \emph{Fekete $n$-tuple} for $K$.
	\end{definition}
	It can be shown that, for any compact set $K$ in $\C$, the sequence $(\delta_n(K))_n$ is decreasing and so converges to a number called the transfinite diameter. It turns out that the transfinite diameter is equal to the capacity, that is, $\lim_{n\to\infty}\delta_n(K)=c(K)$; this is the Fekete-Szeg\"o Theorem.
	
	We now turn our attention to the Dirichlet problem for a domain $D$; that is, for a given continuous function $\phi:\partial D\to\R$, find a harmonic function $h$ on $D$ such that $\lim_{z\to\zeta}h(z)=\phi(\zeta)$ for all $\zeta\in\partial D$. A domain $D$ for which the Dirichlet problem is solvable for any continuous boundary data $\phi$ is called a \emph{regular domain}. As an easy example, we can see that $\D$ is a regular domain. Moreover, if $\phi:\T\to\R$ is a continuous function, then the function $\D\ni z\mapsto\frac{1}{2\pi}\int_{0}^{2\pi}\phi(e^{i\theta})P(z,e^{i\theta})d\theta\in\R$ solves the corresponding Dirichlet problem, where $P$ denotes the Poisson kernel. Roughly speaking, the solution to the Dirichlet problem is given by integrating the boundary data $\phi$ with respect to an appropriate measure. This leads us to the following definition for regular domains:
	\begin{definition}[Harmonic measure]
		Let $D$ be a regular domain. We denote by $\mathcal{B}(\partial D)$ the set of all Borel subsets of $\partial D$. A \emph{harmonic measure} for $D$ is a function $\omega(\cdot,\cdot,D):D\times\mathcal{B}(\partial D)\to\R$ with the following properties:
		\begin{enumerate}
			\item for any $z\in\D$, the set function $\omega(z,\cdot,D)$ is a Borel probability measure on $\mathcal{B}(\partial D)$,
			\item for any continuous function $\phi:\partial D\to\R$, the function
			\[
			\left\{\begin{array}{lll}
				D & \to & \T\\
				z & \mapsto & \int_{\partial D}\phi(\zeta)d\omega(z,\zeta,D)
			\end{array}
			\right.
			\]
			solves the corresponding Dirichlet problem.
		\end{enumerate}
	\end{definition}
	The existence of a harmonic measure for a regular domain is guaranteed by the Riesz Representation Theorem. We note that harmonic measures can be defined differently for more general domains, but the relationship with the Dirichlet problem is lost. In this paper, we are content to define the harmonic measure only for regular domains, as all the domains we will consider are actually simply connected, which are known to be regular. As always, we refer to \cite{Ransford1995} for further details. An important property about the harmonic measure is the following:
	\begin{theorem}
		Let $D$ be a regular domain and $B$ a Borel subset of $\partial D$. Then, the function $D\ni z\mapsto\omega(z,B,D)$ is either a positive harmonic function or the constant function $0$.
	\end{theorem}
	We note that, by a conformal mapping, we can define all the previous notions for subsets of the Riemann sphere $\overline{\C}:=\C\cup\{\infty\}$ instead of $\C$. Thus, we can define the Green function with singularity at $\infty$.
	\begin{definition}[Green function]
		Let $D$ be a proper subdomain of $\overline{\C}$ with $\infty\in D$. A Green function for $D$ with singularity at $\infty$ is a function $g_D(\cdot,\infty):D\to\R$ with the following properties:
		\begin{enumerate}
			\item $g_D(\cdot,\infty)$ is harmonic on $D\setminus\{\infty\}$ and bounded on any neighbourhood of $\infty$,
			\item $g_D(\infty,\infty)=\infty$ and $\lim_{z\to\infty}(g_D(z,\infty)-\log|z|)$ exists in $\R$,
			\item $\lim_{z\to\zeta}g_D(z,\infty)=0$ for every $z\in\partial D$, except possibly for a polar set.
		\end{enumerate}
	\end{definition}
	It is known that a unique Green function exists for a domain $D$, whenever $D$ is a proper subdomain of $\overline{\C}$ with $\infty\in D$ and $\partial D$ non-polar.
	
	\medskip
	
	To conclude this preliminary part, we recall the Bernstein-Walsh Theorem, which is a classical result about quantitative polynomial approximation in the complex plane, and which we wish to generalise. We refer to Theorem 6.3.1 in \cite{Ransford1995}.
	\begin{theoremBW}
		Let $K\subset \C$ be a compact set with connected complement. Let also $U$ be an open set containing $K$. Then, for any $\vp \in A(\overline{U})$,
		\[
		\limsup _{n}\left\{\inf\{\Vert P - \vp \Vert _K:\,P \text{ is a polynomial with deg}(P)\leq n\}\right\}^{1/n}  = \theta < 1,
		\]
		where
		\[
		\theta = \left\{
		\begin{array}{lll}
			\sup_{\overline{\C}\setminus U}e^{-g_{\overline{\C}\setminus K}(z,\infty)} & \text{if } K\text{ is not polar}\\
			0 &  \text{if } K\text{ is polar}
		\end{array}\right..
		\]
	\end{theoremBW}
	
	\subsection{Incomplete polynomial approximation - Proofs of Theorems \hyperref[thmc]{C} and \hyperref[thmd]{D}}\label{wpa}In this section, we first introduce two of the very few results on approximation by incomplete polynomials in the complex plane. Then we shall prove a general statement, namely Theorem \ref{thm-BW-incomplete} below, from which Theorems \hyperref[thmc]{C} and \hyperref[thmd]{D} will immediately follow as corollaries.
	
	
	We recall that approximation by incomplete polynomials means approximation by polynomials of the form
	\[
	P_n(z)=\sum_{k=n/\tau_n}^na_kz^k,
	\]
	where $(\tau_n)_n$ is a sequence of real numbers with $\tau_n >1$, for any $n\in \N$. As said in the introduction, the problem of approximation on compact sets of the complex plane by incomplete polynomials was successfully attacked from the perspective of weighted polynomial approximation. For instance, in the late 1990s, Pritsker and Varga obtained a nice result in this direction. We shall present a particular case of it.
	
	Let $G$ be a bounded domain in $\C$ and $W$ a function holomorphic in $G$ that does not vanish in $G$. Following \cite{PritskerVarga1998}, we say that $(G,W)$ is a pair of approximation if every function $f$ holomorphic in $G$ can be locally uniformly approximated on $G$ by the sequence $(W^nP_n)_n$ for some sequence $(P_n)_n$ of polynomials with $\text{deg}(P_n)\leq n$ for $n\in\N$.
	\begin{theorem}[Theorem 2.2 in \cite{PritskerVarga1998}]\label{thm-PV}Let $G$ be a simply connected bounded domain contained in $\C\setminus (-\infty,0]$ and $W(z)=z^{\alpha}$ for $z\in G$ and some $\alpha >0$ (here we choose the principal branch of the logarithm). Then $(G,z^{\alpha})$ is a pair of approximation if and only if
		\begin{equation}\label{eq-PV}
			\mu:=(1+\alpha)\omega(\infty,\cdot,\overline{\C}\setminus \overline{G}) - \alpha\omega(0,\cdot,\overline{\C}\setminus \overline{G})
		\end{equation}
		is a positive measure. 
	\end{theorem}
	
	For more general weights, see Theorem 1.1 in \cite{PritskerVarga1998}.  Further, one can make the following easy observation, see \cite{CharpentierMaronikolakisII2025}:
	
	\begin{prop}Let $G$ be a simply connected bounded domain contained in $\C\setminus (-\infty,0]$. There exists $\alpha >0$ such that $(G,z^{\alpha})$ has the approximation property.
	\end{prop}
	Having in mind that harmonic measures are conformally invariant and coincide with the Poisson measure if $\Omega=\D$, condition \eqref{eq-PV} is rather easy to check in various concrete situations. We refer to \cite{PritskerVarga1998,CharpentierMaronikolakisII2025} for explicit computations or estimations on simple examples, related to Theorem \ref{thm-PV}. 

	To see why Theorem \ref{thm-PV} is useful to deal with incomplete polynomial approximation, let us fix $\alpha = \frac{p}{q}$ positive and rational. It is easily checked that, if $(G,z^{\alpha})$ has the approximation property then, on any compact set $K \subset G$, any function holomorphic on $G$ can be aproximated uniformly on $K$ by incomplete polynomials of the form $\sum_{k=n/\tau}^n a_k z^k$ with $\tau = \frac{\alpha +1}{\alpha}=1+\frac{q}{p}$. In the same vein as the Bernstein-Walsh Theorem, a quantitative version of Pritsker-Varga's Theorem has recently been formulated in \cite{CharpentierLevenbergWielonsky2024}. It can be stated as follows:
	
	\begin{theorem}\label{thm-PV-quantitative}Let $G$ be a simply connected bounded domain contained in $\C\setminus (-\infty,0]$ and $\alpha >0$ such that $(G,z^{\alpha})$ is a pair of approximation. For any  compact subset $K$ of $G$, with connected complement, there exist $B\geq 1$, $\theta \in (0,1)$ and $N\in \N$, such that for any $\vp\in H(G)$ and any $n\geq N$, there exists a polynomial $P_n$ with $\text{deg}(P_n)\leq n$ such that
		\[
		\Vert z^{\alpha n}P_n - \vp \Vert_K \leq B\Vert \vp\Vert _{\overline{U}}\theta^{n}.
		\]
	\end{theorem}
	
	This result can be seen as a quantitative approximation result by incomplete polynomials of the form $\sum _{k=n/\tau}^na_k z^k$ for a fixed $\tau >1$. However it is not satisfactory enough with respect to Problem \hyperref[qsaip]{QSAIP}. In particular, it does not tell us anything about simultaneous approximation around $0$.
	As the main contribution of this section, the following result will fill this gap. In the sequel, we will call \textit{valuation} of a non-zero polynomial with coefficients $a_k\in\C$, the smallest $k\in \N$ such that $a_k \neq 0$.
	\begin{theorem}\label{thm-BW-incomplete}Let $K\subset \C\setminus \{0\}$ and $L\subset \C$ be two disjoint compact sets with connected complement. Let also $U$ be a bounded open set containing $K$. Then there exist $(\theta_n)_n$ decreasing with $\theta_n \to 1$ as $n\to \infty$, $C>0$ and $G\in (0,1)$ such that, for any $\tau >1$, any $\vp \in A(\overline{U})$, and any $n\in\N$, there exists a polynomial $P_n$ of the form $\sum_{k=n/\tau}^{n} a_k z^k$ such that
		\[
		\max\{\Vert P_n \Vert _L,\Vert \vp - P_n\Vert _K \}\leq C\Vert \vp \Vert _{\overline{U}}\left(\frac{M_{K\cup L}}{m_K/2}\right)^{\frac{n}{\tau}}\left(G\theta_{n-\lfloor \frac{n}{\tau}\rfloor}\right)^{n-\lfloor \frac{n}{\tau}\rfloor},
		\]
		where $M_{K\cup L}=\max_{z\in K\cup L}|z|$ and $m_{K}=\min_{z\in K}|z|$.
	\end{theorem}
	
	\begin{proof}The proof closely follows that of the Bernstein-Walsh Theorem (Theorem 6.3.1 in \cite{Ransford1995}), but requires a careful tracking of the constants and an introduction of the parameter $\tau$ at the right moment.
		
		Upon enlarging $K$, we may and shall assume that $K$ is non-polar. Let $\tau>1$ be fixed. Let $U_1$ and $U_2$ be two disjoint open sets such that $L\subset U_1$ and $K\subset U_2$ (we can assume $U_2\subset U$) with $0\notin \overline{U_2}$. We also fix a path $\Gamma=\Gamma_1\cup \Gamma_2$ such that $\Gamma_1 \subset U_1\setminus L$ and $\Gamma_2 \subset U_2\setminus K$. We assume that both $\Gamma_1$ and $\Gamma_2$ have winding number equal to $1$ around any point of $L$ and $K$, and equal to $0$ around any point of $\C\setminus U_1$ and $\C\setminus U_2$, respectively. Let $\tau>1$ be arbitrary, $\vp \in A(\overline{U})$ and $n\in \N$. We define the function:
		\[
		h_n(z)=\left\{
		\begin{array}{cl}
			\frac{\vp(z)}{z^{\lfloor n/\tau \rfloor}} & \text{if }z\in U_2\\
			0 & \text{if }z\in U_1
		\end{array}
		\right..
		\]
		Let us denote by $q_m$, $m\geq 2$, a Fekete polynomial of degree $m$, for the set $K\cup L$ which is a compact set with connected complement (see Definition 5.5.3 in \cite{Ransford1995}). Let then $Q_m$ be the polynomial
		\[
		Q_m(w)=\frac{1}{2i\pi}\int_{\Gamma}\frac{h_n(z)}{q_m(z)}\frac{q_m(w)-q_m(z)}{w-z}dz.
		\]
		It is easily seen that $Q_m$ is a polynomial of degree at most $m-1$. By Cauchy's integral formula and the definition of $h_n$, we get, for any $w\in K\cup L$,
		\[
		h_n(w)-Q_m(w)=\frac{1}{2i\pi}\int_{\Gamma}\frac{h_n(z)}{w-z}\frac{q_m(w)}{q_m(z)}dz=\frac{1}{2i\pi}\int_{\Gamma_2}\frac{h_n(z)}{w-z}\frac{q_m(w)}{q_m(z)}dz,
		\]
		and thus
		\begin{eqnarray}\label{eqnar1}
			\notag |h_n(w)-Q_m(w)| & \leq & \frac{l(\Gamma_2)}{2\pi}\frac{\Vert h_n \Vert _{\Gamma_2}}{\text{dist}(\Gamma_2,K)}\frac{\Vert q_m \Vert _{K\cup L}}{\min_{z\in \Gamma_2}|q_m(z)|}\\
			& \leq & \frac{l(\Gamma_2)}{2\pi\text{dist}(\Gamma_2,K)}\frac{\Vert \vp\Vert _{\Gamma_2}}{\min_{z\in \Gamma_2}|z|^{\lfloor \frac{n}{\tau} \rfloor}}\frac{\Vert q_m \Vert _{K\cup L}}{\min_{z\in \Gamma_2}|q_m(z)|},
		\end{eqnarray}
		where $l(\Gamma_2)$ stands for the length of $\Gamma_2$. In particular, we get  for any $z\in L$,
		\begin{equation}\label{eqnar2}
			|z^{\lfloor \frac{n}{\tau} \rfloor}Q_m(z)| \leq \frac{l(\Gamma_2)\Vert \vp\Vert _{\Gamma_2}}{2\pi\text{dist}(\Gamma_2,K)}\frac{\max_{z\in L}|z|^{\lfloor \frac{n}{\tau} \rfloor}}{\min_{z\in \Gamma_2}|z|^{\lfloor \frac{n}{\tau} \rfloor}}\frac{\Vert q_m \Vert _{K\cup L}}{\min_{z\in \Gamma_2}|q_m(z)|}
		\end{equation}
		and for any $z\in K$,
		\begin{equation}\label{eqnar3}
			|z^{\lfloor \frac{n}{\tau} \rfloor}Q_m(z)-\vp(z)|\leq  \frac{l(\Gamma_2)\Vert \vp\Vert _{\Gamma_2}}{2\pi\text{dist}(\Gamma_2,K)}\frac{\max_{z\in K}|z|^{\lfloor \frac{n}{\tau} \rfloor}}{\min_{z\in \Gamma_2}|z|^{\lfloor \frac{n}{\tau} \rfloor}}\frac{\Vert q_m \Vert _{K\cup L}}{\min_{z\in \Gamma_2}|q_m(z)|}.
		\end{equation}
		Now, since $K\cup L$ is non-polar, by Bernstein's lemma (Theorem 5.5.7 (b) in \cite{Ransford1995}),
		\begin{equation}\label{eqnar4}
		\frac{\Vert q_m \Vert _{K\cup L}}{\min_{z\in \Gamma_2}|q_m(z)|}\leq \max_{z\in \Gamma_2}e^{-mg_{\overline{\C}\setminus (K\cup L)}(z,\infty)}\left(\frac{\delta_m(K\cup L)}{c(K\cup L)}\right)^{m\max_{z\in \Gamma_2}d_{\overline{\C}\setminus (K\cup L)}(z,\infty)},
		\end{equation}
		where $d_{\overline{\C}\setminus (K\cup L)}$ is the Harnack distance for the domain $\overline{\C}\setminus (K\cup L)$ and $g_{\overline{\C}\setminus (K\cup L)}(\cdot,\infty)$ is the Green function for $\overline{\C}\setminus (K\cup L)$ at $\infty$. Let us set
		\[
		G= \max_{z\in \Gamma_2}e^{-g_{\overline{\C}\setminus (K\cup L)}(z,\infty)}\quad \text{and}\quad \theta_m=\left(\frac{\delta_m(K\cup L)}{c(K\cup L)}\right)^{\max_{z\in \Gamma_2}d_{\overline{\C}\setminus (K\cup L)}(z,\infty)}.
		\]
		We have $G\in (0,1)$ and, by Fekete-Szegö theorem, $(\theta_m)_m$ decreases to $1$ as $m\to \infty$ (see Theorem 5.5.2 in \cite{Ransford1995}).
		
		It follows from \eqref{eqnar1}, \eqref{eqnar2}, \eqref{eqnar3} and \eqref{eqnar4}, with $m=n-\lfloor\frac{n}{\tau}\rfloor$, that if we set $P_n=z^{\lfloor\frac{n}{\tau}\rfloor}Q_{n-\lfloor\frac{n}{\tau}\rfloor}$, then
		\[
		\Vert P_n \Vert _L \leq C\Vert \vp\Vert _{\overline{U}}\left(\frac{\max_{z\in L}|z|}{\min_{z\in \Gamma_2}|z|}\right)^{n/\tau}(G\theta_{n-\lfloor\frac{n}{\tau}\rfloor})^{n-\lfloor\frac{n}{\tau}\rfloor}
		\]
		and
		\[
		\Vert \vp - P_n \Vert _K \leq C\Vert \vp\Vert _{\overline{U}}\left(\frac{\max_{z\in K}|z|}{\min_{z\in \Gamma_2}|z|}\right)^{n/\tau}(G\theta_{n-\lfloor\frac{n}{\tau}\rfloor})^{n-\lfloor\frac{n}{\tau}\rfloor},
		\]
		where $C>0$ is independent of $\tau$. Also observe that $\lfloor n/\tau \rfloor\leq \text{val}(P_n)\leq \text{deg}(P_n)\leq n$. The conclusion follows, choosing $\Gamma_2$ such that $\min_{z\in \Gamma_2}|z|\geq m_K/2$.
	\end{proof}
	
	\begin{remark}\label{remark-useful}\quad \begin{enumerate}\item From the proof above, we notice that the constants $C$ and $G$ in the statement of Theorem \ref{thm-BW-incomplete} can be chosen as
			\[
			C=\frac{l(\Gamma)}{2\pi\text{dist}(\Gamma,K)}\quad \text{and}\quad G= \max_{z\in \Gamma}e^{-g_{\overline{\C}\setminus (K\cup L)}(z,\infty)},
			\]
			where $g_{\overline{\C}\setminus (K\cup L)}(z,\infty)$ is the Green function of $\overline{\C}\setminus (K\cup L)$ at $\infty$, and $\Gamma$ is any Jordan curve contained in $U$, that contains $K$ in its interior and $L$ in its exterior.
			
			\item Comparing Theorem \ref{thm-BW-incomplete} with Theorem \ref{thm-PV-quantitative}, note that the compact set $K$ does not need to be a subset of $\C\setminus (-\infty,0]$.
			
		\end{enumerate}
	\end{remark}

	Now, it is quite easy to derive Theorems \hyperref[thmc]{C} and \hyperref[thmd]{D} from Theorem \ref{thm-BW-incomplete}. Indeed, let us observe that one can choose $n$ large enough so that
	\[
	C^{1/n}\left(G\theta_{n-\lfloor\frac{n}{\tau}\rfloor}\right)^{\frac{n-\lfloor\frac{n}{\tau}\rfloor}{n}} \in (0,1)
	\]
	for any $\tau \geq 2$. Then $\tau$ can also be chosen large enough so that $\left(\frac{M_{K\cup L}}{m_K/2}\right)^{\frac{1}{\tau}}$ is so close to $1$ that 
	\[
	C^{1/n}\left(\frac{M_{K\cup L}}{m_K/2}\right)^{\frac{1}{\tau}}\left(G\theta_{n-\lfloor\frac{n}{\tau}\rfloor}\right)^{\frac{n-\lfloor\frac{n}{\tau}\rfloor}{n}}\in (0,1).
	\]
	Note that the choice of $n$ is independent of $\tau$ and that $\tau$ only depends on $M_{K\cup L}$, $m_K$, the Green function of $\overline{\C}\setminus (K\cup L)$, and the choice of $\Gamma$ (which depends only on $K$ and $L$).

	For the readers' convenience, and for further use in the next section, let us restate Theorems \hyperref[thmc]{C} and \hyperref[thmd]{D}, as corollaries.
	\begin{cor}\label{cor-incomplete-1}Let $K\subset \C\setminus \{0\}$ and $L\subset \C$ be two disjoint compact sets with connected complement. Let also $U$ be a bounded open set containing $K$. Then there exist $\tau \geq 1$, $\theta\in (0,1)$ and $N\in \N$ such that for any $\vp \in A(\overline{U})$ and any $n \geq N$, there exists a polynomial $P_n$ of the form $\sum_{k= n/\tau}^{n} a_k z^k$ such that
		\[
		\Vert P_n \Vert _L \leq \Vert \vp \Vert _{\overline{U}}\theta^{n}\quad \text{and}\quad \Vert P_n - \vp \Vert _K \leq \Vert \vp \Vert _{\overline{U}}\theta^{n}.
		\]
	\end{cor}
	
	\quad
	
	\begin{cor}\label{cor-incomplete}Let $(\tau_n)_n$ be a sequence of positive real numbers with $\tau _n \to +\infty$. For any disjoint compact sets $K\subset \C\setminus \{0\}$ and $L\subset \C$ , with connected complement, and for any bounded open set $U$ containing $K$, there exist $\theta\in (0,1)$ and $N\in \N$ such that, for any $\vp \in A(\overline{U})$ and any $n \geq N$, there exists a polynomial $P_n$ of the form $\sum_{k=n/\tau_n}^{n}a_kz^k$ such that
		\[
		\Vert P_n \Vert _L \leq \Vert \vp \Vert _{\overline{U}}\theta^{n}\quad \text{and}\quad \Vert P_n - \vp\Vert _K \leq \Vert \vp \Vert _{\overline{U}}\theta^{n}.
		\]
	\end{cor}
	
	Let us end this section by showing to which extent the above two corollaries are optimal. Firstly, a reformulation of Corollary 2.3 in \cite{PritskerVarga1998} in terms of incomplete polynomials, yields the following.
	\begin{prop}Let $D(a,r) \subset \C \setminus \{0\}$ be a non-empty disc. If any polynomial $\vp$ can be approximated uniformly on $\overline{D(a,r)}$ by polynomials of the form $\sum_{k= n/\tau}^{n} a_k z^k$, then $r\leq a\frac{\tau -1}{\tau +1}$.
	\end{prop}
	Thus, if we choose $K$ with non-empty interior in Corollary \ref{cor-incomplete-1}, then there exists a $\tau_K >1$ such that, for any $1<\tau < \tau_K$, the set of all incomplete polynomials of the form $\sum_{k= n/\tau}^{n} a_k z^k$ is not dense in $A(K)$. This shows that polynomials of the form $\sum_{k= n/\tau}^{n} a_k z^k$ in Corollary \ref{cor-incomplete} cannot be replaced by polynomials of the form $\sum_{k= n/\tau_n}^{n} a_k z^k$, with $\tau_n$ decreasing to $1$. 
	
	The same type of argument can be used to show that Corollary \ref{cor-incomplete} is optimal. Indeed, using Proposition 2.4 in \cite{CharpentierMaronikolakisII2025}, we can show the following: given any $\tau >1$, there exists a compact set $K\subset \C\setminus \{0\}$, with non-empty interior, such that the set of all polynomials of the form $\sum_{k= n/\tau}^{n} a_k z^k$ is not dense in $A(K)$. Therefore, the condition $\tau _n \to +\infty$ in Corollary \ref{cor-incomplete} cannot be replaced by $(\tau_n)_n$ bounded.
	
	\medskip
	
	At this point of our study, we are naturally led to the problem of approximation by incomplete polynomials of the form $\sum_{k= n/\tau_n}^{n} a_k z^k$, with $\tau_n \to 1$. We shall discuss it in the next subsection.
	
	\subsection{Approximation by strongly incomplete polynomials - Proof of Theorem \hyperref[thme]{E}}
	
	In the sequel, we will refer to incomplete polynomials of the form $\sum_{k= n/\tau_n}^{n} a_k z^k$, with $\tau_n \to 1$, as \textit{strongly} incomplete polynomials. We have not found in the literature research dealing with approximation by strongly incomplete polynomials, although, according to \cite{SaffVarga1979,Golitschek1980}, it is most likely that any sequence of strongly incomplete polynomials that would be bounded on $\D$ would converge to $0$ at least in some neighbourhood of $0$. However, we are able to show that approximation by strongly incomplete polynomials is possible when $K$ is very \textit{small}, namely, when $K$ is polar.
	
	This paragraph is dedicated to proving Theorem \hyperref[thme]{E}. For convenience, let us recall its statement.
	
	\begin{theoremE}Let $K\subset \C\setminus \{0\}$ be a polar compact set and let $U$ be a bounded open set containing $K$. There exist a sequence $(\tau_n)_n$ in $(1,\infty)$ with $\tau_n\to1$, a sequence $(\eta_n)_n$ of positive real numbers with $\eta_n\to 0$ and $N\in \N$, such that for any $\vp\in A(\overline{U})$ and any $n\geq N$, there exists a polynomial $P_n$ of the form $\sum_{k= n/\tau_n}^{n}a_kz^k$ such that
		\[
		\Vert P_n-\vp\Vert_K\leq \Vert \vp\Vert _{\overline{U}}\eta_n^{n}.
		\]
	\end{theoremE}
	For the proof of this theorem, we will need preliminary results that might also be of independent interest. We refer to Section \ref{pot} for the definitions of capacity and energy. In the sequel, for a non-empty compact set $K\subset \C$ and for $\varepsilon>0$, we will denote by $K^{\varepsilon}$ the closure of the $\varepsilon$-neighbourhood of $K$,
	\[
	K^{\varepsilon}:=\left\{z\in\C:\,\text{dist}(z,K)\leq\varepsilon\right\}.
	\]
	\begin{lemma}\label{dilation}
		Let $K\subseteq\C$ be a non-empty compact set. For any $\varepsilon_1>\varepsilon_2>0$, one has
		\[
		c(K^{\varepsilon_1})\leq2\frac{\varepsilon_1}{\varepsilon_2}c(K^{\varepsilon_2}).
		\]
	\end{lemma}
	\begin{proof}
		Let $\varepsilon>0$. Clearly, $K^{\varepsilon_1}\subseteq\cup_{z\in K}D(z,\varepsilon_1+\varepsilon)$, so, by compactness, there exists $N\in\N$ and distinct points $z_1,\dots,z_N\in K$ such that
		\[
		K^{\varepsilon_1}\subseteq\cup_{k=1}^ND(z_k,\varepsilon_1+\varepsilon)\subseteq\cup_{k=1}^N\overline{D(z_k,\varepsilon_1+\varepsilon)}=:\tilde{K}.
		\]
		Since $\tilde{K}$ is compact, there exists an equilibrium measure $\mu$ for $\tilde{K}$. We set
		\[
		\mu_1=\mu_{|\overline{D(z_1,\varepsilon_1+\varepsilon)}}\quad\text{and}\quad\mu_k=\mu_{|\overline{D(z_k,\varepsilon_1+\varepsilon)}\setminus\bigcup_{i=1}^{k-1}\overline{D(z_i,\varepsilon_1+\varepsilon)}},\,k=2,\dots,N. \]
		We then have that $\mu=\sum_{k=1}^{N}\mu_k$. Moreover, $\mu$ is actually supported on $\partial\tilde{K}$ so, for $k=1,\dots,N$, $\mu_k$ is supported on $C_k:=\partial D(z_k,\varepsilon_1+\varepsilon)\setminus \left(\cup_{i< k} \overline{D(z_i,\varepsilon_1+\varepsilon)}\cup\cup_{i> k} D(z_i,\varepsilon_1+\varepsilon)\right)$.
		
		Let $k=1,\dots,N$ and let $T_k:\overline{D(z_k,\varepsilon_1+\varepsilon)}\to\overline{D(z_k,\varepsilon_2)}$ be the function defined by
		\[
		T_k(z)=\frac{\varepsilon_2}{\varepsilon_1+\varepsilon}z+\left(1-\frac{\varepsilon_2}{\varepsilon_1+\varepsilon}\right)z_k,\,z\in\overline{D(z_k,\varepsilon_1+\varepsilon)}.
		\]
		Let $\nu_k$ be the push-forward measure of $\mu_k$ by $T_k$, that is, $\nu_k(B)=\mu_k\left(T_k^{-1}(B)\right)$ for any Borel measurable subset $B$ of $\overline{D(z_k,\varepsilon_2)}$. Since $T_k$ is invertible, we can say that $\mu_k$ is the push-forward measure of $\nu_k$ by $T_k^{-1}$, where $T_k^{-1}(z)=\frac{\varepsilon_1+\varepsilon}{\varepsilon_2}z-\left(\frac{\varepsilon_1+\varepsilon}{\varepsilon_2}-1\right)z_k$ for $z\in\overline{D(z_k,\varepsilon_2)}$. Let also $C_k'=T_k(C_k)$.
		
		\medskip
		
		Let us now estimate the energy $I(\mu)$ of $\mu$:
		\begin{align*}
			\begin{autobreak}
				\MoveEqLeft[0]
				I(\mu)=\int_{\tilde{K}}\int_{\tilde{K}}\log|z-w|d\mu(z)d\mu(w)
				=\int_{\tilde{K}}\int_{\tilde{K}}\log|z-w|d\left(\sum_{k=1}^{N}\mu_k\right)(z)d\left(\sum_{k=1}^{N}\mu_k\right)(w)
				=\sum_{k=1}^N\int_{C_k}\int_{C_k}\log|z-w|d\mu_k(z)d\mu_k(w)
				+2\sum_{1\leq k<j\leq N}\int_{C_j}\int_{C_k}\log|z-w|d\mu_k(z)d\mu_j(w)
			\end{autobreak}
		\end{align*}
		First, for $k=1,\dots,N$, we have
		\begin{align*}
			\begin{autobreak}
				\MoveEqLeft[0]
				\int_{C_k}\int_{C_k}\log|z-w|d\mu_k(z)d\mu_k(w)
				=\int_{C_k'}\int_{C_k'}\log|T_k^{-1}(z')-T_k^{-1}(w')|d\nu_k(z')d\nu_k(w')
				=\int_{C_k'}\int_{C_k'}\log\left|\frac{\varepsilon_1+\varepsilon}{\varepsilon_2}z'-\left(\frac{\varepsilon_1+\varepsilon}{\varepsilon_2}-1\right)z_k-\frac{\varepsilon_1+\varepsilon}{\varepsilon_2}w'+\left(\frac{\varepsilon_1+\varepsilon}{\varepsilon_2}-1\right)z_k\right|d\nu_k(z')d\nu_k(w')
				=\int_{C_k'}\int_{C_k'}\log\left|z'-w'\right|d\nu_k(z')d\nu_k(w')+\log\left(\frac{\varepsilon_1+\varepsilon}{\varepsilon_2}\right)\nu_k(C_k')^2.
			\end{autobreak}
		\end{align*}
		Next, for $1\leq k<j\leq N$, we similarly have
		\begin{align*}
			\begin{autobreak}
				\MoveEqLeft[0]
				\int_{C_j}\int_{C_k}\log|z-w|d\mu_k(z)d\mu_j(w)
				=\int_{C_j'}\int_{C_k'}\log|T_k^{-1}(z')-T_k^{-1}(w')|d\nu_k(z')d\nu_k(w')
				=\int_{C_j'}\int_{C_k'}\log\left|\frac{\varepsilon_1+\varepsilon}{\varepsilon_2}z'-\left(\frac{\varepsilon_1+\varepsilon}{\varepsilon_2}-1\right)z_k-\frac{\varepsilon_1+\varepsilon}{\varepsilon_2}w'+\left(\frac{\varepsilon_1+\varepsilon}{\varepsilon_2}-1\right)z_j\right|d\nu_k(z')d\nu_j(w')
				=\int_{C_j'}\int_{C_k'}\log\left|\frac{\varepsilon_1+\varepsilon}{\varepsilon_2}(z'-w')-\left(\frac{\varepsilon_1+\varepsilon}{\varepsilon_2}-1\right)(z_k-z_j)\right|d\nu_k(z')d\nu_j(w')
				=\int_{C_j'}\int_{C_k'}\log\left|z'-w'\right|d\nu_k(z')d\nu_j(w')+\log\left(\frac{\varepsilon_1+\varepsilon}{\varepsilon_2}\right)\nu_k(C_k')\nu_j(C_j')
				+\int_{C_j'}\int_{C_k'}\log\left|1-\left(1-\frac{\varepsilon_2}{\varepsilon_1+\varepsilon}\right)\frac{z_k-z_j}{z'-w'}\right|d\nu_k(z')d\nu_j(w')
				\leq \int_{C_j'}\int_{C_k'}\log\left|z'-w'\right|d\nu_k(z')d\nu_j(w')+\log\left(\frac{\varepsilon_1+\varepsilon}{\varepsilon_2}\right)\nu_k(C_k')\nu_j(C_j')
				+\int_{C_j'}\int_{C_k'}\log\left(1+\left(1-\frac{\varepsilon_2}{\varepsilon_1+\varepsilon}\right)\frac{|z_k-z_j|}{|z'-w'|}\right)d\nu_k(z')d\nu_j(w').
			\end{autobreak}
		\end{align*}
		Let us fix $1\leq k<j\leq N$, $z'\in C_k'$ and $w'\in C_j'$. We shall estimate the quantity $\frac{|z_k-z_j|}{|z'-w'|}$. Let $z\in C_k$ and $w\in C_j$ be such that $z'=\frac{\varepsilon_2}{\varepsilon_1+\varepsilon}z+\left(1-\frac{\varepsilon_2}{\varepsilon_1+\varepsilon}\right)z_k$ and $w'=\frac{\varepsilon_2}{\varepsilon_1+\varepsilon}w+\left(1-\frac{\varepsilon_2}{\varepsilon_1+\varepsilon}\right)z_j$. Since the sets $C_k$ and $C_j$ are contained in different half-planes created by the perpendicular bisector of the line segment $[z_k,z_j]$, we have $\langle z-w,z_k-z_j\rangle>0$ (where $\langle\cdot,\cdot\rangle$ denotes the standard inner product on $\R^2$). Therefore
		\begin{align*}
			|z'-w'|^2 & =\left|\frac{\varepsilon_2}{\varepsilon_1+\varepsilon}z+\left(1-\frac{\varepsilon_2}{\varepsilon_1+\varepsilon}\right)z_k-\frac{\varepsilon_2}{\varepsilon_1+\varepsilon}w-\left(1-\frac{\varepsilon_2}{\varepsilon_1+\varepsilon}\right)z_j\right|^2\\
			& =\left|\frac{\varepsilon_2}{\varepsilon_1+\varepsilon}(z-w)+\left(1-\frac{\varepsilon_2}{\varepsilon_1+\varepsilon}\right)(z_k-z_j)\right|^2\\
			& =\left(\frac{\varepsilon_2}{\varepsilon_1+\varepsilon}\right)^2\left|z-w\right|^2+\left(1-\frac{\varepsilon_2}{\varepsilon_1+\varepsilon}\right)^2\left|z_k-z_j\right|^2\\
			& +2\left(\frac{\varepsilon_2}{\varepsilon_1+\varepsilon}\right)\left(1-\frac{\varepsilon_2}{\varepsilon_1+\varepsilon}\right)\langle z-w,z_k-z_j\rangle
			\geq\left(1-\frac{\varepsilon_2}{\varepsilon_1+\varepsilon}\right)^2\left|z_k-z_j\right|^2,
		\end{align*}
		so we get
		$$\left(1-\frac{\varepsilon_2}{\varepsilon_1+\varepsilon}\right)\frac{|z_k-z_j|}{|z-w|}\leq1.$$
		
		We deduce that
		\begin{multline*}
			\int_{C_j'}\int_{C_k'}\log\left(1+\left(1-\frac{\varepsilon_2}{\varepsilon_1+\varepsilon}\right)\frac{|z_k-z_j|}{|z'-w'|}\right)d\nu_k(z')d\nu_j(w')
			\leq\int_{C_j'}\int_{C_k'}\log(2)d\nu_k(z')d\nu_j(w')\\
			=\log(2)\nu_k(C_k')\nu_j(C_j'),
		\end{multline*}
		and thus
		\begin{multline*}
			\int_{C_j}\int_{C_k}\log|z-w|d\mu_k(z)d\mu_j(w)\\
			\leq\int_{C_j'}\int_{C_k'}\log\left|z'-w'\right|d\nu_k(z')d\nu_j(w')+\log\left(2\frac{\varepsilon_1+\varepsilon}{\varepsilon_2}\right)\nu_k(C_k')\nu_j(C_j').
		\end{multline*}
		Putting everything together, we obtain
		\begin{align*}
			\begin{autobreak}
				\MoveEqLeft[0]
				I(\mu)=\sum_{k=1}^N\int_{C_k}\int_{C_k}\log|z-w|d\mu_k(z)d\mu_k(w)
				+2\sum_{1\leq k<j\leq N}\int_{C_j}\int_{C_k}\log|z-w|d\mu_k(z)d\mu_j(w)
				\leq\sum_{k=1}^N\int_{C_k'}\int_{C_k'}\log\left|z'-w'\right|d\nu_k(z')d\nu_k(w')+2\sum_{1\leq k<j\leq N}\int_{C_j'}\int_{C_k'}\log\left|z'-w'\right|d\nu_k(z')d\nu_j(w')
				+\sum_{k=1}^N\log\left(\frac{\varepsilon_1+\varepsilon}{\varepsilon_2}\right)\nu_k(C_k')^2+2\sum_{1\leq k<j\leq N}\log\left(2\frac{\varepsilon_1+\varepsilon}{\varepsilon_2}\right)\nu_k(C_k')\nu_j(C_j')
				\leq I(\nu)+\log\left(2\frac{\varepsilon_1+\varepsilon}{\varepsilon_2}\right).
			\end{autobreak}
		\end{align*}
		Taking exponentials, we get that
		$$e^{I(\mu)}\leq2\frac{\varepsilon_1+\varepsilon}{\varepsilon_2}e^{I(\nu)}.$$
		Since $K^{\varepsilon_1}\subseteq\tilde{K}$ and since the measure $\nu$ is supported on $K^{\varepsilon_2}$, we finally obtain
		$$c(K^{\varepsilon_1})\leq c(\tilde{K})=e^{I(\mu)}\leq2\frac{\varepsilon_1+\varepsilon}{\varepsilon_2}e^{I(\nu)}\leq 2\frac{\varepsilon_1+\varepsilon}{\varepsilon_2}c(K^{\varepsilon_2}).$$
		Letting $\varepsilon\to0$, we get the desired result.
	\end{proof}
	
	Using the Fekete-Szeg\"o Theorem, it is not difficult to see that $\lim_n \delta_n\left(K^{\frac{1}{n}}\right) =\lim_n c\left(K^{\frac{1}{n}}\right)=c(K)$, for any compact set $K\subset \C$. In particular, if $K$ is not polar, we have
	\[
	\lim_{n\to\infty}\frac{\delta_n\left(K^{\frac{1}{n}}\right)}{c\left(K^{\frac{1}{n}}\right)}=1.
	\]
	More generally, we can prove the following.
	\begin{lemma}\label{conv}
		Let $K\subseteq\C$ be a compact set. Then,
		\[
		\limsup_{n\to\infty}\frac{\delta_n\left(K^{\frac{1}{n}}\right)}{c\left(K^{\frac{1}{n}}\right)}\leq2.
		\]
	\end{lemma}
	\begin{proof}
		For $n\geq2$, observe that $K^{\frac{1}{n-1}}=\left\{z\in\C:dist\left(z,K^{\frac{1}{n}}\right)\leq\frac{1}{n-1}-\frac{1}{n}\right\}$. Proceeding as in the proof of the Fekete-Szeg\"o Theorem (Theorem 5.5.2 in \cite{Ransford1995}), we can show that
		\[
		\left(\frac{1}{n-1}-\frac{1}{n}\right)^{\frac{1}{n-1}}\delta_n\left(K^{\frac{1}{n}}\right)\leq c\left(K^{\frac{1}{n-1}}\right)^{\frac{n}{n-1}}.
		\]
		This implies $\delta_n\left(K^{\frac{1}{n}}\right)\leq\left(n(n-1)\right)^{\frac{1}{n-1}}c\left(K^{\frac{1}{n-1}}\right)^{\frac{n}{n-1}}$, and thus
		\begin{equation*}
			\frac{\delta_n\left(K^{\frac{1}{n}}\right)}{c\left(K^{\frac{1}{n}}\right)}\leq\left(n(n-1)\right)^{\frac{1}{n-1}}\frac{c\left(K^{\frac{1}{n-1}}\right)^{\frac{n-1}{n}}}{c\left(K^{\frac{1}{n}}\right)}=\left(n(n-1)\right)^{\frac{1}{n-1}}\frac{c\left(K^{\frac{1}{n-1}}\right)}{c\left(K^{\frac{1}{n}}\right)}\frac{1}{c\left(K^{\frac{1}{n-1}}\right)^{\frac{1}{n}}}.
		\end{equation*}
		We shall now apply Lemma \ref{dilation} twice: for $\varepsilon_1=\frac{1}{n-1}$ and $\varepsilon_2=\frac{1}{n}$ we get
		\begin{equation*}
			\frac{c\left(K^{\frac{1}{n-1}}\right)}{c\left(K^{\frac{1}{n}}\right)}\leq2\frac{n}{n-1}
		\end{equation*}
		and, for $\varepsilon_1=1$ and $\varepsilon_2=\frac{1}{n-1}$, we get $c\left(K^1\right)\leq2(n-1)c\left(K^{\frac{1}{n-1}}\right)$, hence
		\begin{equation*}
			\frac{1}{c\left(K^{\frac{1}{n-1}}\right)^{\frac{1}{n}}}\leq\left(2\frac{n-1}{c\left(K^1\right)}\right)^{\frac{1}{n}}.
		\end{equation*}
		Combining the previous three inequalities, we obtain
		\[
		\frac{\delta_n\left(K^{\frac{1}{n}}\right)}{c\left(K^{\frac{1}{n}}\right)}\leq2\left(n(n-1)\right)^{\frac{1}{n-1}}\frac{n}{n-1}\left(2\frac{n-1}{c\left(K^1\right)}\right)^{\frac{1}{n}},
		\]
		and the result follows by taking the $\limsup$.
	\end{proof}
	
	\begin{proof}[Proof of Theorem E]
		As usual, by Mergelyan's Theorem, we can assume that $\vp$ is a polynomial. For $n\geq 1$,  let us simply write $K_n=K^{\frac{1}{n}}$. Let $n_0\in \N$ be fixed such that, for any $n\geq n_0$, the compact set $K_n$ has connected complement and is contained in $U\cap \C\setminus \{0\}$. We also fix a path $\Gamma$ such that $\Gamma\subset U\setminus K_{n_0}$. We assume that $\Gamma$ has winding number equal to $1$ around any point of $K_{n_0}$, equal to $0$ around any point of $\C\setminus U$ and is such that $\min_{z\in \Gamma}|z|\geq m_{K_{n_0}}/2$ where $m_{K_{n_0}}:=\min_{z\in K_{n_0}}|z|$.
		
		Let $n\geq n_0$, then we set
		\[
		G_n=\max_{z\in \Gamma}e^{-g_{\overline{\C}\setminus K_n}(z,\infty)}\quad \text{and}\quad \theta_n=\left(\frac{\delta_n(K_n)}{c(K_n)}\right)^{\max_{z\in \Gamma}d_{\overline{\C}\setminus K_n}(z,\infty)}.
		\]
		Since $\max_{z\in \Gamma}d_{\overline{\C}\setminus K_n}(z,\infty)\leq\max_{z\in \Gamma}d_{\overline{\C}\setminus K_{n_0}}(z,\infty)$ (Corollary 1.3.7 in \cite{Ransford1995}), we have, by Lemma \ref{conv},
		\[
		\theta:=\sup_{n}\theta_n < \infty.
		\]
		Moreover, as explained in the second part of the proof of Theorem 6.3.1 in \cite{Ransford1995}, one has $G_n \to 0$. Hence, upon replacing $n_0$ by a larger integer, we may and shall assume that $G_n < 1/\theta$ for any $n\geq n_0$.
		
		We now apply Theorem \ref{thm-BW-incomplete} with $K=K_n$ and $L=\emptyset$. For every $\tau >1$ and every $n\geq n_0$, there exists a polynomial $P_n$ of the form $\sum_{k= n/\tau}^{n} a_k z^k$ such that
		\begin{eqnarray}\label{main_ineq}
			\Vert \vp - P_n\Vert _{K_n} & \leq & C_n\Vert \vp \Vert _{\overline{U}}\left(\frac{M_{K_{n_0}}}{m_{K_{n_0}}/2}\right)^{\frac{n}{\tau}}\left(G_n\theta_{n-\lfloor \frac{n}{\tau}\rfloor}\right)^{n-\lfloor \frac{n}{\tau}\rfloor}\nonumber\\
			 &  \leq & C\Vert \vp \Vert _{\overline{U}}\left(\left(\frac{M_{K_{n_0}}}{m_{K_{n_0}}/2}\right)^{\frac{1}{\tau}}\left(\theta G_n\right)^{\frac{\tau-1}{\tau}}\right)^n,
		\end{eqnarray}
		where $M_{K_{n_0}}=\max_{z\in K_{n_0}}|z|$, $C_n=\frac{l(\Gamma)}{2\pi\text{dist}(\Gamma,K_n)}$ and $C:=C_{n_0}$. 
		We now set
		\[
		\tau_n:=\frac{\log(\theta G_n)}{\log(\theta G_n)+\sqrt{-\log(\theta G_n)}}\left(1-\frac{\log\left(\frac{M_{K_{n_0}}}{m_{K_{n_0}}/2}\right)}{\log(\theta G_n)}\right)+\frac{1}{n}
		\]
		and
		\[
		\eta_n:=C^{\frac{1}{n}}\left(\frac{M_{K_{n_0}}}{m_{K_{n_0}}/2}\right)^{\frac{1}{\tau_n}}(\theta G_n)^{\frac{\tau_n-1}{\tau_n}}.
		\]
		We easily check that $\tau_n \to 1$. Moreover, using that $G_n < 1/\theta$, a straightforward computation shows that the inequality
		\[
		\tau_n>\frac{\log(\theta G_n)}{\log(\theta G_n)+\sqrt{-\log(\theta G_n)}}\left(1-\frac{\log\left(\frac{M_{K_{n_0}}}{m_{K_{n_0}}/2}\right)}{\log(\theta G_n)}\right)
		\]
		is equivalent to
		\[
		\left(\frac{M_{K_{n_0}}}{m_{K_{n_0}}/2}\right)^{\frac{1}{\tau_n}}(\theta G_n)^{\frac{\tau_n-1}{\tau_n}}<e^{-\sqrt{-\log(\theta G_n)}},
		\]
		and so
		\[
		\eta_n<C^{\frac{1}{n}}e^{-\sqrt{-\log(\theta G_n)}}
		\]
		Therefore $\eta_n \to 0$
		which, after taking into account \eqref{main_ineq} completes the proof.
	\end{proof}
	\begin{remark}
		We note that, by Theorem \hyperref[thme]{E}, the rate of approximation by strongly incomplete polynomials on polar compact sets is (strictly) faster than geometric. This is in accordance with the Bernstein-Walsh Theorem for polar compact sets  (see (6.8) in Theorem 6.3.1 in \cite{Ransford1995}).
	\end{remark}
	
	\subsection{Open problems}
	
	The previous results raise several natural problems. We shall mention two of them. Firstly, according to Lemma \ref{dilation}, we can ask whether the following improvement holds.
	\begin{conjecture}For any non-empty compact set $K$,
		\[
		\lim_{n\to\infty}\frac{\delta_n\left(K^{\frac{1}{n}}\right)}{C\left(K^{\frac{1}{n}}\right)}=1.
		\]
	\end{conjecture}
	We recall that, by the Fekete-Szeg\"o Theorem, this holds if $K$ is not polar.
	
	Furthermore, to us, the problem of characterizing the compact sets that support approximation by strongly incomplete polynomials, is quite interesting.  In view of Theorem \hyperref[thme]{E}, we make the following conjecture.
	
	\begin{conjecture}\label{conj-polar}Let $K\subset \C\setminus \{0\}$ be a compact set. Then any function continuous on $K$ (or equivalently, any polynomial) can be uniformly approximated on $K$ by polynomials of the form $\sum_{k= n/\tau_n}^{n}a_kz^k$ with $\tau_n \to 1$, if and only if $K$ is polar.
	\end{conjecture}
	
	As far as we know, it is not even known whether strongly incomplete polynomials are dense in $C(K)$ when $K$ is an arc of a circle or an interval in $\C\setminus \{0\}$. If this conjecture is true, then one can also wonder to which extent the rate of convergence of $\tau_n \to 1$ is related to $K$. In the opposite direction, is it possible to approximate \textit{anything} on any polar set by strongly incomplete polynomial with respect to any sequence $(\tau_n)_n$ with $\tau _n \to 1$?
	
	\section{Frequently universal Taylor series}\label{sec-FUTS}
	
	\subsection{Weighted densities and frequently universal Taylor series}\label{prelim-1-1}
	For a sequence $\alpha=(\alpha_k)_{k\geq 1}$ of non-negative real numbers such that $\sum_{k\geq 1}\alpha_k=+\infty$, we define the matrix $(\alpha_{n,k})_{n,k\geq 1}$ by 
	\[
	\alpha_{n,k}=\left\{
	\begin{array}{cl}\alpha_k/\sum_{j=1}^n\alpha_j& \text{for }1\leq k\leq n,\\
		0 & \text{otherwise}.
	\end{array}\right.
	\]
	Then the lower $\alpha$-density $\dlow_{\alpha}:\PP(\N)\to [0,1]$ and the upper $\alpha$-density $\dup_{\alpha}:\PP(\N)\to [0,1]$ are defined, for any $E\subset \N$, by
	\[
	\dlow_{\alpha}(E)=\liminf_{n \to \infty}\sum_{k\geq 1}\alpha_{n,k}\indicatrice{E}(k)\quad \text{and}\quad \dup_{\alpha}(E)=1-\dlow_{\alpha}(\N\setminus E)=\limsup_{n \to \infty}\sum_{k=1}^{+\infty}\alpha_{n,k}\indicatrice{E}(k).
	\]
	If $\alpha_k=1$ for any $k\geq 1$, the density $\dlow_{\alpha}$ is the natural density (that is simply denoted by $\dlow$). For more information on those densities, we refer to \cite{FreedmanSember1981}. If it is assumed in addition that the sequence $(\alpha_n/\sum_{j=1}^n\alpha_j)_{n\geq 1}$ converges to $0$, then for any set $E\subset \N$ enumerated as an increasing sequence $(n_k)_{k\geq 1}$, we have
	\[
	\dlow_{\alpha}(E)=\liminf_{k \to \infty}\frac{\sum_{j=1}^{k}\alpha_{n_j}}{\sum_{j=1}^{n_k}\alpha_j},
	\]
	see Lemma 2.7 in \cite{ErnstMouze2019}. From now on, we say that a sequence $\alpha=(\alpha_k)_{k\geq 1}$ of non-negative numbers is \textit{admissible} if it satisfies the conditions $\alpha_n/\sum_{j=1}^n\alpha_j\to 0$ as $n\to \infty$ and $\sum_k\alpha_k =+\infty$.
	
	\medskip
	
	Lemma 2.8 in \cite{ErnstMouze2019} shows that densities can be scaled in function of the growth of the sequences $\alpha$. More precisely, for two admissible sequences $\alpha$ and $\beta$, let us write $\alpha \lesssim \beta$ if there exists $k_0\in \N$ such that $(\alpha_k/\beta_k)_{k\geq k_0}$ is non-increasing. Then  $\alpha \lesssim \beta$ implies that for any $E\subset \N$,
	\[
	\dlow_{\beta}(E)\leq \dlow_{\alpha}(E)\leq \dup_{\alpha}(E) \leq \dup_{\beta}(E).
	\]
	In particular, if $\alpha$ is non-increasing and $\beta$ non-decreasing, then $\dlow_{\beta}\leq \dlow \leq \dlow_{\alpha}$.
	
	\medskip
	
	
	\begin{example}\label{examples-densities}{\rm The interested reader can find several examples of admissible sequences $\alpha$, inducing weighted densities in \cite{ErnstMouze2019,ErnstMouze2021}, see also \cite{CharpentierErnstMestiriMouze2022}. For our purpose, we shall consider the following ones:
			\begin{enumerate}
				\item The classical natural density $\dlow$ is equal to $\dlow_{\alpha}$ where $\alpha$ is any constant sequence $(a,a,a,\ldots)$ with $a>0$.
				\item The logarithmic density $\dlow _{\log}$ is $\dlow_{\alpha}$ for the sequence $\alpha = (1/k)_{k\geq 1}$.
				\suspend{enumerate}
				The following sequences are also admissible sequences that appear in the literature:
				\resume{enumerate}
				\item $\PP_r=(k^r)_{k\geq 1}$ for $r> -1$.
				\item $\EE_{\varepsilon}=(\exp(k^{\varepsilon}))_{k\geq 1}$ for $\varepsilon\in [0,1]$.
				\item $\DD_s=(\exp(k/\log_{(s)}(k)))_{k\geq k_0}$ for $s\in \N\cup \{\infty\}$ and $k_0$ large enough. Here we denote $\log_{(s)}=\log \circ \cdots \circ \log$, $\log$ appearing $s$ times, and use the conventions $\log_{(0)}(x)=x$ and $\log_{(\infty)}(x)=1$ for any $x>0$.
				\item $\LL_l=(e^{\log(k)\log_{(l)}(k)})_{k\geq k_0}$ for $l\geq 1$ and $k_0$ large enough.
		\end{enumerate}}
	\end{example}
	On can check that for any $0<\varepsilon < 1$, any $s\in \N$, any $r> 0$, any $t\in (-1,0)$ and any $l\geq 1$, we have
	\begin{equation}\label{scale-density}
		\dlowg{\EE_1} \leq \dlowg{\DD_s} \leq \dlowg{\EE_{\varepsilon}} \leq \dlow_{\LL_l}\leq\dlowg{\PP_r} \leq \dlow \leq \dlowg{\PP_t} \leq \dlow_{\log}.
	\end{equation}
	
	Further, given two admissible sequences $\alpha$ and $\beta$, the densities $\dlow_{\alpha}$ and $\dlow_{\beta}$ are said to be equivalent if, for any subset $E\subset \N$, one has
	\[
	\dlow_{\alpha}(E)>0\text{ if and only if }\dlow_{\beta}(E)>0.
	\]
	For example:
	\begin{prop}[Lemma 2.10 in \cite{ErnstMouze2019}]\label{prop-equiv-Pr-natural}For any $r>-1$, the densities $\dlow_{\PP_r}$ and $\dlow$ or equivalent.
	\end{prop}
	In fact, under some mild condition, the densities that are equivalent to the natural one can be characterised. Given an admissible sequence $\alpha$, let us denote by $\vp_{\alpha}$ the function defined by $\vp_{\alpha}(x)=\sum_{k=1}^x\alpha_k$, $x\geq 1$. We shall say that the function $\vp_{\alpha}$ satisfies the condition $\Delta_2$ if there exists $K>0$ such that $\vp_{\alpha}(2x)\leq K\vp_{\alpha}(x)$, for some $K>0$ and any $x$ large enough. Then, one has:
	\begin{prop}[Theorem 4.2 (1) in \cite{CharpentierErnstMestiriMouze2022}]\label{prop-equi-density-delta2}Let $\alpha$ be a non-decreasing admissible sequence. Assume that $(\frac{\vp_{\alpha}(n)}{\vp_{\alpha}(Cn)})_n$ is eventually non-increasing, for some $C>1$. Then $\dlow_{\alpha}$ and $\dlow$ are equivalent if and only if $\vp_{\alpha}$  satisfies the condition $\Delta_2$.
	\end{prop}
	
	\begin{proof}Since $\alpha$ is non-decreasing, it is clear that, for any $E\subset \N$, $\dlow_{\alpha}(E)>0$ implies $\dlow(E)>0$. The fact that if $\vp_{\alpha}$ satisfies the condition $\Delta_2$, then $\dlow(E)>0$ implies $\dlow_{\alpha}(E)>0$ for any $E\subset \N$, is contained in the proof of Theorem 4.2 (1) in \cite{CharpentierErnstMestiriMouze2022}. Conversely, using that  $(\frac{\vp_{\alpha}(n)}{\vp_{\alpha}(Cn)})_n$ is eventually non-increasing, for any $C>1$, it is not difficult to see that, if $\vp_{\alpha}$ does not satisfy the condition $\Delta_2$, then $\frac{\vp_{\alpha}(n)}{\vp_{\alpha}(Cn)}\to 0$, for some $C>1$. Then, let $a>C$ and let us consider the two sets
		\[
		E=\N\cap \bigcup_{n\in \N} [a^{2n},C a^{2n}]\quad \text{and}\quad F=\N\cap \bigcup_{n\in \N}[a^{2n+1},C a^{2n+1}].
		\]
		Using $a>C$, one directly checks that $E\cap F=\emptyset$. Therefore,
		\begin{eqnarray*}
			\dlow_{\alpha}(E) & \leq & 1-\dup_{\alpha}(F)\\
			& \leq & 1 - \limsup_n \left(1 - \frac{\vp_{\alpha}(a^{2n+1})}{\vp_{\alpha}(Ca^{2n+1})}\right)\\
			& = & 0.
		\end{eqnarray*}
		However, $\dlow(E)\geq \liminf_n \frac{(C-1)a^{2n}}{a^{2n+2}}=\frac{C-1}{a^2}>0$, which concludes the proof.
	\end{proof}
	
	Note that, for any admissible weight $\alpha$ within the classes $\EE_{\varepsilon}$, $\DD_s$, $\LL_l$ and $\PP_r$, $r> 1$, of Example \ref{examples-densities}, the ratio $(\frac{\vp_{\alpha}(n)}{\vp_{\alpha}(Cn)})_n$ is eventually non-increasing, for some $C>1$ (and actually for any $C>1$). Moreover, among these examples, the only non-decreasing sequences $\alpha$ for which $\vp_{\alpha}$ satisfies the condition $\Delta_2$ are the sequences $\PP_r$, $r\geq 0$.
	
	\medskip
	
	To finish this paragraph, let us recall the definition of the main objects of interest to us, and the motivating problems.
	\begin{definition}[Frequently universal Taylor series]Let $\alpha$ be an admissible sequence and let $K\subset \C\setminus \D$ be a compact set. A function $f=\sum_ka_kz^k$ in $H(\D)$ is a $K$-frequently universal Taylor series with respect to $\alpha$ if, for any non-empty open set $U$ of $A(K)$,
		\[
		\dlow_{\alpha}\left(\{n\in \N:\,S_n^{(K)}(f) \in U\}\right)>0.
		\]
		The set of $K$-frequently universal Taylor series with respect to $\alpha$ is denoted by $FU_{\alpha}(\D,K)$.
	\end{definition}
	When $\alpha$ is the constant sequence, we shall simply write $FU(\D,K)$, as we did in Definition \ref{definition-intro-K-FUTS}, and refer to its elements as $K$-frequently universal Taylor series. We recall that
	\begin{equation}\label{eq-result-no-FUTS-Nesto}
		\bigcap_{K\in \MM}FU(\D,K)=\emptyset,
	\end{equation}
	see \cite{MouzeMunnier2016,MouzeMunnier2017,Charpentier2019}. In view of Proposition \ref{prop-equiv-Pr-natural}, this implies, more generally,
	\[
	\bigcap_{K\in \MM}FU_{\alpha}(\D,K)=\emptyset
	\]
	whenever $\PP_r  \lesssim\alpha$ for some $r>-1$.
	
	\medskip
	
	Thus, we are led to the following reformulations of Questions A and B of the introduction:
	\begin{enumerate}
		\item Is the set $\cap_{K\in \MM}FU_{\log}(\D,K)$ empty or not, where $FU_{\log}=FU_{\alpha}$ with $\alpha =(1/k)_k$?
		\item Given a fixed compact set $K\subset \C\setminus \D$, does there exist an admissible sequence $\alpha$ such that $FU_{\alpha}(\D,K)\neq \emptyset$?
	\end{enumerate}
	
	In \cite{MouzeMunnier2016}, Question (1) corresponds to Problem (2), while Question (2) was posed for the natural density as Problem (3).
	
	\subsection{Main results - Proofs of Theorems \hyperref[thma]{A} and \hyperref[thmb]{B}}\label{subsec-mainresults}
	
	We shall first prove the existence of frequently universal Taylor series with respect to arbitrary compact subsets of $\C\setminus \overline{\D}$, with connected complement, that is Theorem \hyperref[thmb]{B}. As usual, we will denote by $A(\D)$ the disc algebra, that is, the space of all continuous functions on $\overline{\D}$ that are holomorphic in $\D$ endowed with the supremum norm.
	
	\begin{theorem}\label{FHC-Taylor-series}For any compact set  $K \subset \C\setminus \overline{\D}$, with connected complement, the set $FU(\D,K)\cap A(\D)$ is dense in $A(\D)$.
	\end{theorem}
	
	Since $A(\D)$ is dense in $H(\D)$ (with respect to the topology of locally uniform convergence), the conclusion of the theorem holds true if one replaces $A(\D)$ with $H(\D)$. Moreover, by a Baire argument, the set of $K$-frequently universal Taylor series is meagre in $A(\D)$ and in $H(\D)$.

	Like most of the constructive proofs of the existence of frequently universal vectors for sequences of operators, the proof makes use of the existence of countably many disjoint sets with positive lower density, satisfying suitable conditions. In our case, we will need the following lemma.
	
	\begin{lemma}[Lemma 2.2 in \cite{CharpentierErnstMestiriMouze2022}]\label{lemma-FHC-sets}Let us fix $\kappa>1$, $\nu\in \N\setminus\{0\}$, and an increasing sequence $(N_p)_p$ of positive integers. There exists a family $(E_p)_p$ of increasing sequences of multiples of $\nu$, such that:
		\begin{enumerate}
			\item for any $p\in\N$, the set $E_p$ has positive lower density;
			\item for any $p\neq q\in\N$ and any $(n,m)\in E_p\times E_q$ with $n>m$, one has $n> \kappa m$;
			\item for any $p\in\N$, we have $\min\{k\in E_p\}\geq N_p$.
		\end{enumerate}
	\end{lemma}
	
	\begin{proof}[Proof of Theorem \ref{FHC-Taylor-series}]Since polynomials are dense in $A(\D)$, and since $f+P$ clearly belongs to $FU(\D,K)$ whenever $f$ belongs to $FU(\D,K)$ and $P$ is a polynomial, it is enough to build a function in $FU(\D,K)$ that is in $A(\D)$ and then the density of such functions would follow.
		
		Let $K$ be a compact subset of $\C\setminus \overline{\D}$ with connected complement and $U$ be a bounded open set containing $K$ such that $\overline{\D}\cap U=\emptyset$.		
		Let $(\vp_p,r_p)_p$ be an enumeration of all pairs $(\vp,r)$ where $\vp$ runs over a set of polynomials that is dense in $A(K)$, and where $r$ runs over $\Q^+:=\Q\cap\{x>0\}$. We may and shall assume that
		\begin{equation}\label{eq-vp}
			\Vert \vp_p \Vert _{\overline{U}} \leq \frac{p}{4},\quad p\in \N.
		\end{equation}
		By Mergelyan's theorem, the countable set of non-empty open balls $\left\{U_p:=B(\vp_p,r_p):\,p\in \N\right\}$ is a basis for the topology of $A(K)$. To prove the theorem, it is therefore enough to prove the existence of $f\in A(\D)$ with the property that for any $p\in \N$,
		\[
		\dlow\left(\{n\in \N:\, S_n^{(K)}(f) \in U_p\}\right)>0.
		\]
		
		Then,  let $\tau\geq 2$, $\theta\in (0,1)$ and $N\in \N$ be given by  Corollary \ref{cor-incomplete-1} applied with $K$ and $U$ as above, and $L=\overline{\D}$. In the sequel, we may and shall assume that $\tau$ also satisfies
		\begin{equation}\label{eq-assumpt-tau-FHC}
			G_U^{1/\tau}\theta \in (0,1),
		\end{equation}
		where $G_U=\max_{z\in \overline{U}}e^{g_{\overline{\C}\setminus \overline{\D}}(z,\infty)}$ and $g_{\overline{\C}\setminus \overline{\D}}(\cdot,\infty)$ is the Green function of $\overline{\C}\setminus \overline{\D}$ at $\infty$.
		
		Let us now define a sequence $(N_p)_p$ of integers such that for any $p\geq 0$, $N_p\geq N$, and for any $n\geq N_p$,
		\begin{equation}\label{eq-Np}
			n(G_U^{1/\tau}\theta)^{n}< \min\left\{\frac{1}{2},r_p\right\}.
		\end{equation}
		Note that the latter is possible by the choice of $\tau$ in \eqref{eq-assumpt-tau-FHC}.
		
		Finally, we fix a sequence $(E_p)_p$ of disjoint subsets of $\N$ with positive lower density, given by Lemma \ref{lemma-FHC-sets} applied  to the sequence $(N_p)_p$ and to $\kappa=\tau+1$. Let us denote by
		$(s_n)_n$ an increasing enumeration of $\cup_p E_p$. For every $n\in \N$, we denote by $p_n\in \N$ the (only) integer such that $s_n \in E_{p_n}$.

		We will build by induction on $n$ a sequence of polynomials $P_n$ with degree at most $s_n$, $n\in \N$. The function $f$ will then be written as the sum $\sum _n P_n$. Let us start with $P_0$ and let $p_0\in \N$ be such that $s_0\in E_{p_0}$. Note that $s_0 \geq N_{p_0}\geq N$. Thus, by Corollary \ref{cor-incomplete-1}, there exists a polynomial $P_0$ of the form $\sum_{k= s_0/\tau +1}^{s_0}a_kz^k$ that satisfies, using $\Vert \vp_{p_0} \Vert_{\overline{U}} \leq s_0$,
		\begin{equation}\label{eq1-FHC-proof}
			\Vert P_0  \Vert _{\overline{\D}} \leq s_0\theta^{s _0}\quad \text{and}\quad \Vert P_0 - \vp_{p_0} \Vert _K \leq s_0\theta^{s _0}.
		\end{equation}
		We will now build the polynomials $P_n$, $n\geq 1$. Let us assume that $P_0,\ldots,P_{n-1}$ have been built, for some $n\geq 1$. In order to define $P_n$, we distinguish two cases. If $p_n=p_{n-1}$ (meaning that $s_n$ and $s_{n-1}$ both belongs to $E_{p_{n-1}}=E_{p_n}$), then we set $P_n \equiv 0$. If not, we will use again Corollary \ref{cor-incomplete-1} to choose $P_n$ satisfying the following properties:
		\begin{enumerate}[(a)]
			\item \label{item1-FHC-proof}$\lfloor s_n/\tau \rfloor +1\leq \text{val}(P_n)\leq \text{deg}(P_n)\leq s_n$;
			\item \label{item2-FHC-proof}$\Vert P_n\Vert _{\overline{\D}} \leq s_n(G_U^{1/\tau}\theta)^{s _n}$;
			\item \label{item3-FHC-proof}$\left\Vert P_n - \left(\vp_{p_n} - \sum_{j=0}^{n-1}P_j\right)\right\Vert _K \leq s_n(G_U^{1/\tau}\theta)^{s _n}$.
		\end{enumerate}
		More precisely, since $N_{p_n}\geq  N$, Corollary \ref{cor-incomplete-1} applied to $\vp=\vp_{p_n}-\sum_{j=0}^{n-1}P_j$ yields a polynomial $P_{n}$ satisfying \eqref{item1-FHC-proof} and such that
		\begin{eqnarray}\label{eq2-FHC-proof}
			\notag \max\left\{\Vert P_n\Vert _{\overline{\D}},\left\Vert P_n - \left(\vp_{p_n}-\sum_{j=0}^{n-1}P_j\right)\right\Vert _K\right\} & \leq & \left\Vert\vp_{p_n} - \sum_{j=0}^{n-1}P_j\right\Vert _{\overline{U}}\theta^{s_n}\\
			& \leq & \left(\frac{s_n}{4} + \left\Vert \sum_{j=0}^{n-1}P_j\right\Vert _{\overline{U}}\right)\theta^{s_n}.
		\end{eqnarray}
		Now, by Bernstein's lemma (see Theorem 5.5.7 (a) in \cite{Ransford1995}), we get
		\[
		\left\Vert \sum_{j=0}^{n-1}P_j\right\Vert _{\overline{U}} \leq \left\Vert \sum_{j=0}^{n-1}P_j\right\Vert _{\overline{\D}}G_U^{s_{n-1}}\leq \frac{n}{2}G_U^{\frac{s_n}{\tau}},
		\]
		where we used the induction hypothesis, \eqref{eq-Np}, $G_U\geq 1$ and $s_{n-1}< s_n/\tau$ (which holds true since $E_{p_{n-1}}\neq E_{p_n}$).
		The previous estimate together with \eqref{eq2-FHC-proof} yields
		\[
		\max\left\{\Vert P_n\Vert _{\overline{\D}},\left\Vert P_n - \left(\vp_{p_n}-\sum_{j=0}^{n-1}P_j\right)\right\Vert _K\right\}\leq \frac{s_n}{4}\theta^{s_n} + \frac{s_n}{2}\left(G_U^{1/\tau}\theta\right)^{s_n},
		\]
		hence \eqref{item2-FHC-proof} and \eqref{item3-FHC-proof} (using again $G_U\geq 1$).
		
		Setting $f=\sum_n P_n$, it is clear by \eqref{item2-FHC-proof} that $f\in A(\D)$. Let us check that the partial sums of $f$ visit frequently any $U_p$, $p\in \N$. Let us fix $p\in \N$ and $l\in E_p$. We let $m:=\min\{n\in \N:\,\{s_{n},s_{n+1},\ldots,l\}\subset E_p\}$.  Let us remind that for any $u,v\in \N$ and $p\neq q$ such that $s_v\in E_q$, $s_u\in E_p$ and $s_{v}>s_{u}$, one has $s_{v}>(\tau +1)s_{u}$. Therefore, we deduce from the constructions of the $P_n$ (and in particular from items \eqref{item1-FHC-proof} and \eqref{item3-FHC-proof}) that
		\[
		\left\Vert S_l(f) - \vp_p \right\Vert _K = \left\Vert \sum_{j=0}^{m}P_j - \vp _p\right\Vert _K \leq s_m(G_U^{1/\tau}\theta)^{s _m}< r_p,
		\]
		where the last inequality comes from \eqref{eq-Np}. We conclude that $S_l(f)\in U_p$ for any $l\in E_p$.

	\end{proof}

	\begin{remark}\label{remark-papad}The proof of Theorem \ref{FHC-Taylor-series} gives us a bit more than the statement. It actually tells us that there exists $f\in A(\D)$ such that, for any compact set $K\subset \C\setminus \overline{\D}$, with connected complement, and any $p\in \N$, there exists an increasing sequence $(n_k)_k$, with positive density, such that $S_{n_k}(f) \to \vp_p$ uniformly on $K$. We mention that, as noticed in \cite{Papachristodoulos2013}, the previous cannot hold replacing the countable family of all $\vp_p$, $p\in \N$, by the set of all polynomials.
	\end{remark}
	
	Let now $\alpha =(\alpha_n)_n$ be an admissible sequence. We immediately deduce the following corollary from Theorem \ref{FHC-Taylor-series} and Proposition \ref{prop-equi-density-delta2}.
	\begin{cor}\label{cor-alpha-FUTS}Let $\alpha$ be an admissible sequence. Assume that $\alpha  \lesssim \beta$ for some non-decreasing sequence $\beta$, for which $\vp_{\beta}$ satisfies the condition $\Delta_2$. Then, for any compact set $K\subset \C\setminus \overline{\D}$ with connected complement, the set $FU_{\alpha}(\D,K)$ is non-empty.
	\end{cor}
	
	We shall now prove that the $\Delta_2$ condition is sharp, under the extra condition that the quotient sequence $(\frac{\vp_{\alpha}(n)}{\vp_{\alpha}(Cn)})_n$ is eventually non-increasing for any $C>1$.
	
	\begin{prop}\label{thm-non-FUTS-alpha-fast}Let $\alpha$ be an admissible weight such that $(\frac{\vp_{\alpha}(n)}{\vp_{\alpha}(Cn)})_n$ is eventually non-increasing for any $C>1$, and  assume that $\vp_{\alpha}$ does not satisfy the $\Delta_2$ condition. Let also $K\subset \C\setminus \overline{\D}$ be a compact set, with connected complement. If $FU_{\alpha}(\D,K)$ is non-empty, then $K$ has the property that any polynomial can be uniformly approximated on $K$ by incomplete polynomials of the form $\sum_{k=\frac{n}{\tau_n}}^{n} a_k z^k$, for some non-increasing sequence $(\tau_n)_n$ with $\tau_n \to 1$. In particular, $K$ has empty interior.
	\end{prop}
	
	For the proof, we will need the following easy lemma.
	
	\begin{lemma}\label{lemma-non-FU-alpha}Let $\alpha$ be an admissible weight such that $(\frac{\vp_{\alpha}(n)}{\vp_{\alpha}(Cn)})_n$ is non-increasing for any $C>1$. Let $(n_k)_k$ be an increasing sequence of positive intergers. If $\vp_{\alpha}$ does not satisfy the condition $\Delta _2$, then, for any $C>1$ and any $k_0\in \N$,
		\[
		\dup_{\alpha}\left(\N\cap \bigcup_{k \geq k_0} \left[\lfloor \frac{n_k}{C}\rfloor,n_k \right]\right) =1.
		\]
	\end{lemma}
	
	\begin{proof}By assumption, $\frac{\vp_{\alpha}(n)}{\vp_{\alpha}(Cn)} \to 0$ as $n\to \infty$, for any $C>1$. Then the verification of the lemma is straightforward, since
		\[
		\dup_{\alpha}\left(\N\cap \bigcup_{k \geq k_0} \left[\lfloor \frac{n_k}{C}\rfloor,n_k \right]\right)\geq 1 - \lim_k \frac{\vp_{\alpha}(n_k/C)}{\vp_{\alpha}(n_k)} =1.
		\]
	\end{proof}
	
	\begin{proof}[Proof of Proposition \ref{thm-non-FUTS-alpha-fast}]We shall first prove the first assertion of the statement. Assume that $f=\sum_ka_kz^k$ belongs to $FU_{\alpha}(\D,K) \neq \emptyset$. Let $\vp$ be any non-zero polynomial, and let $\varepsilon_n>0$ with $\varepsilon_n\to 0$ as $n\to \infty$. We may and shall assume that $B(\vp,\varepsilon_n)\cap B(0,\varepsilon_n)=\emptyset$, $n\in \N$, where $B(x,r)$ denotes the ball of $A(K)$ centered at $x$ with radius $r$.
		
		For any $n\in \N$, there exists $E_n \subset \N$ such that $\dlow_{\alpha}(E_n)>0$ and $S_l(f)\in B(\vp,\varepsilon_n)$ for every $l\in E_n$. Similarly, there exist $F_n \subset \N$ with $\dlow_{\alpha}(F_n)>0$ such that $S_l(f)\in B(0,\varepsilon_n)$, $n\in \N$, $l\in F_n$. By assumption, $E_n\cap F_n =\emptyset$, $n\in \N$, hence by Lemma \ref{lemma-non-FU-alpha}, for any $C>1$ and any $n\in \N$, there exists $l\in E_n$ and $k\in F_n$ such that $k<l$ and $k\in \{\lfloor \frac{l}{C}\rfloor,\ldots, l\}$. In particular we can build two increasing sequences $(l_n)_n$ and $(k_n)_n$ such that $\lim l_n/k_n=1$ and, for any $n\in \N$,
		\[
		\Vert S_{l_n}(f) - \vp \Vert _{K} \leq \varepsilon_n\quad \text{and}\quad 	\Vert S_{k_n}(f) \Vert _{K} \leq \varepsilon_n.
		\]
		In particular, for any $n\in \N$,	$(\sum_{i=k_n}^{l_n}a_iz^i)_n$ converges to $\vp$ uniformly on $K$, which gives the first part of the theorem.
		
		For the last assertion, let us assume by contradiction that $K$ has non empty interior, and let $D(a,r)$ be a non-empty open disc contained in $K$. We easily deduce from Corollary 2.3 in \cite{PritskerVarga1998} that there exists $\tau_0 >1$ such that for any $\tau \in (1,\tau_0)$, there exist compact sets $K' \subset D(a,r)$ such that the set of all incomplete polynomials of the form $\sum_{k=n/\tau}^na_kz^k$ is not dense in $A(K')$, whence that no set of incomplete polynomials of the form $\sum_{k=n/\tau_n}^{n}a_kz^k$ with $\tau_n\to 1$ can be dense in $A(K')$.
	\end{proof}
	
	\begin{example}{\rm The previous results show that, among the classes of weighted densities given in Example \ref{examples-densities}, $FU_{\alpha}(\D,K)$ is non-empty for some $K\subset \C\setminus \overline{\D}$, with connected complement and non-empty interior, if and only if $\alpha \in \PP_r$, for some $r> -1$, or $\alpha=(1/k)_{k\geq 1}$ (logarithmic density).
		}
	\end{example}
	
	\begin{remark}{\rm Proposition \ref{thm-non-FUTS-alpha-fast} exhibits a strong contrast with Ernst and Mouze's result asserting that sequences of operators that satisfy the frequent universality criterion are frequently universal with respect to all reasonable weighted densities \cite{ErnstMouze2019,ErnstMouze2021}. In particular, for compact sets $K$ with non-empty interior, this shows that the sequence $(S_n^{(K)})_n$ does not satisfy the frequent universality criterion, although it is universal.}
	\end{remark}
	
	Let us now focus on the case of the logarithmic density and prove Theorem \hyperref[thma]{A}. The logarithmic density offers a bit more flexibility than the natural density, as it allows to choose sets $E_p$ with $\dlow_{\log}(E_p)>0$, that are increasingly separated from each other. This allows us to use Corollary \ref{cor-incomplete} instead of Corollary \ref{cor-incomplete-1}. Let us restate Theorem restate \hyperref[thma]{A}. We recall that we denote by $\MM$ the set of all compact sets in $\C\setminus \D$, with connected complement.
	
	\begin{theorem}\label{thm-log-density-FUTS}The set $\bigcap_{K\in \MM}FU_{\log}(\D,K)$ is dense in $H(\D)$.
	\end{theorem}
	
	To prove Theorem \ref{thm-log-density-FUTS}, we shall first state a lemma analogous to Lemma \ref{lemma-FHC-sets}, that is more or less already contained in \cite{ErnstMouze2021}. For the sake of clarity and completeness, we will give some details of this proof.
	
	\begin{lemma}\label{lemma-log-FHC-sets}Let us fix a family $(N_p(i))_p$, $i\in \N$, of increasing sequences of positive integers. There exists a family $(E_p(i))_{p}$, $i\in \N$, of sequences of subsets of $\N$ such that:
		\begin{enumerate}
			\item for any $p,i\in\N$, the set $E_{p}(i)$ has positive $\log$-density;
			\item for any $(p,i)\neq (q,j)\in\N\times\N$ and any $(n,m)\in E_p(i)\times E_q(j)$ with $n>m$, one has $n> m^2$;
			\item for any $p,i\in\N$, we have $\min\{k\in E_p(i)\}\geq N_p(i)$.
		\end{enumerate}
	\end{lemma}
	
	\begin{proof}The proof is a combination of that of  Lemma 4.2 and Lemma 4.3 in \cite{ErnstMouze2021}, and Lemma 2.2 in \cite{CharpentierErnstMestiriMouze2022}. Let $a>1$ and $0<\varepsilon<1/5$ be given such that 
		\begin{equation}\label{eq1-lemma-FHC-log}
			a^2>2\frac{1+\varepsilon}{1-\varepsilon}.
		\end{equation}
		Let $A_p(i)$, $p,i\in \N$, be pairwise disjoints subsets of $\N$ with bounded gaps (that is, for any $p,i\in \N$, there exists $M_{p}(i)>0$ such that for any $n<m$ consecutive in $A_{p}(i)$, one has $n-m\leq M_{p}(i)$). Upon removing finitely many elements of each set $A_{p}(i)$, we may and shall assume that for any $u\in A_{p}(i)$,
		\begin{equation}\label{eq2-lemma-FHC-log}
			2^{(1+\varepsilon)a^{2u}}>2N_p(i).
		\end{equation}
		Now, for any $u\in \N$, we set
		\[
		I_{u,a}^{\varepsilon} = \left[2^{(1-\varepsilon) a^{2u}},2^{(1+\varepsilon) a^{2u}}\right].
		\]
		Using \eqref{eq2-lemma-FHC-log} and $\varepsilon<1/5$, we easily get that, for any $p,i\in \N$ and any $u\in A_p(i)$, the  length of the interval $I_{u,a}^{\varepsilon}$ is greater than $2N_p(i)$, and then that $I_{u,a}^{\varepsilon}$ contains at least one multiple of $N_p(i)$.
		
		Finally, let us define, for any $(p,i)\in \N\times \N$,
		\[
		E_p(i):=\bigcup_{u\in A_p(i)}I_{u,a}^{\varepsilon} \cap N_p(i)\N,
		\]
		where $N_p(i)\N$ denotes the set of all multiples of $N_p(i)$.
		It is clear that, for any $p,i\in \N$, the set $E_p(i)$ satisfies (3). Further, the proof of \cite[Lemma 4.3]{ErnstMouze2019} yields $\underline{d}_{\log}(E_p(i))>0$, that is (1). To check (2), let us fix $(p,i)\neq (q,j)$ and any $(n,m)\in E_p(i)\times E_q(j)$ with $n>m$. There exist $u>v$ such that $n\in I_{u,a}^{\varepsilon}$ and $m\in I_{v,a}^{\varepsilon}$. Then
		\[
		m^2\leq 2^{2(1+\varepsilon)a^{2v}}\quad \text{and} \quad n\geq 2^{(1-\varepsilon)a^{2u}}
		\]
		whence
		\[
		\frac{n}{m^2}\geq 2^{(1-\varepsilon)a^{2u} - 2(1+\varepsilon)a^{2v}} = 2^{(1-\varepsilon)a^{2v}\left(a^{2(u-v)}-2\frac{1+\varepsilon}{1-\varepsilon}\right)}  \geq 1,
		\]
		where the last inequality holds true since $a^2\geq 2\frac{1+\varepsilon}{1-\varepsilon}$ by \eqref{eq1-lemma-FHC-log}.
		
	\end{proof}
	
	Let us now turn to the proof of Theorem \ref{thm-log-density-FUTS}.
	
	\begin{proof}[Proof of Theorem \ref{thm-log-density-FUTS}]Since $f+P$ belongs to $\bigcap_{K\in \MM}FU_{\log}(\D,K)$ whenever $f$ belongs to $\bigcap_{K\in \MM}FU_{\log}(\D,K)$ and $P$ is a polynomial, it is enough to prove that $\bigcap_{K\in \MM}FU_{\log}(\D,K)$ is non-empty.
		
		Let $(\vp_p,r_p)_p$ be an enumeration of all pairs $(\vp,r)$ where $\vp$ runs over the set of polynomials whose coefficients have rational coordinates (note that this set is dense in $A(K)$ for any $K\in\MM$), and where $r$ runs over $\Q^+:=\Q\cap\{x>0\}$. Let $(K_i)_i$ be a sequence of compact subsets of $\C\setminus \D$, with connected complement, such that any compact set $K\subset \C\setminus \D$ that has connected complement, is contained in some $K_i$. The construction of such a sequence $(K_i)_i$ is standard in the theory of universality; we refer to Lemma 2.1 in \cite{Nestoridis1996}. Let $U_{i}\subset\C\setminus D(0,2/3)$ be a bounded open set containing $K_i$, $i\in \N$. For any $i\in \N$, let also $(L^i_n)_n$ be an increasing exhaustion of $\D$ such that $L^i_0 =\overline{D(0,1-\frac{1}{i+2})}$.
		For every $i,n\in \N$, let $\theta_{i,n} \in (0,1)$ and $N_{i,n}\in \N$ be given by Corollary \ref{cor-incomplete}, applied to $K_i$, $L_n^i$ and $U_i$. We define a sequence $(N_p(i))_p$, $i\in \N$, such that for any $i,p\in\N$ and any $n\geq N_p(i)$, one has
		\begin{equation}\label{eq-assumpt-log-0}
			N_p(i) \geq \max\left\{\Vert \vp_p \Vert_{\overline{U_i}},N_{i,0}\right\} \quad \text{and}\quad 2nG_i^{\sqrt{n}}\theta_{i,0}^{n}\leq \frac{r_p}{n^2},
		\end{equation}
		where $G_i=\sup_{z\in \overline{U_i}}e^{g_{\overline{\C}\setminus \overline{D(0,1/2)}}(z,\infty)}$ and $g_{\overline{\C}\setminus \overline{D(0,1/2)}}(\cdot,\infty)$ stands for the Green function of $\overline{\C}\setminus \overline{D(0,1/2)}$ at $\infty$.
		
		Upon enumerating $(L^i_n)_n$ in such a way that each $L^i_n$ is repeated enough times (finitely many), we may and shall assume that for every $i,p\in \N$ and every $n\geq 0$,
		\begin{equation}\label{eq-assumpt-log-1}
			N_p(i)+n \geq N_{i,n} \quad\text{and}\quad 2(N_p(i)+n)G_i^{\sqrt{N_p(i)+n}}\theta_{i,n}^{N_p(i)+n} \leq \frac{r_p}{(N_p(i)+n)^2}.
		\end{equation}

		\medskip
		
		Then let $E_p(i)$, $i,p\in \N$, be pairwise disjoint subsets of $\N$ with positive logarithmic densities, given by Lemma \ref{lemma-log-FHC-sets}, applied for the above family of sequences $(N_p(i))_p$, $i\in \N$. Let us denote by
		$(s_n)_n$ an increasing enumeration of $\cup_{i,p} E_p(i)$. For every $n\in \N$, we denote by $(i_n,p_n)\in \N^2$ the (unique) pair of integers such that $s_n \in E_{p_n}(i_n)$.
		
		As in the proof of Theorem \ref{FHC-Taylor-series}, the desired function $f$ will be written as a sum of polynomials $P_n$, $n\in \N$, that we shall now build by induction on $n$. Note that $s_0\in E_{p_0}(i_0)$ satisfies $s_0 \geq N_{p_0}(i_0)\geq \max\{p, N_{i_0,0}\}$.	Thus, by Corollary \ref{cor-incomplete}, there exists a polynomial $P_0$ with $\sqrt{s_0}\leq \text{val}(P_0)\leq \text{deg}(P_0)\leq s_0 $ such that 
		\begin{equation*}
			\Vert P_0 \Vert _{L^{i_0}_{0}}\leq \Vert \vp_{p_0} \Vert _{\overline{U_{i_0}}}\theta_{i_0,0}^{s_0}\quad \text{and} \quad\Vert P_0 - \vp_{p_0} \Vert _K \leq \Vert \vp_{p_0} \Vert _{\overline{U_{i_0}}}\theta_{i_0,0}^{s_0}.
		\end{equation*}
		Note that, by \eqref{eq-assumpt-log-0}, the previous inequalities become
		\begin{equation}\label{eq1-FHC-proof}
			\Vert P_0 \Vert _{L^{i_0}_{0}}\leq N_{p_0}(i_0)\theta_{i_0,0}^{s_0}\leq \frac{r_{p_0}}{s_0^2}\quad \text{and} \quad\Vert P_0 - \vp_{p_0} \Vert _K \leq N_{p_0}(i_0) \theta_{i_0,0}^{s_0}\leq \frac{r_{p_0}}{s_0^2}.
		\end{equation}
		We will now build the polynomials $P_n$, $n\geq 1$. Let us assume that $P_0,\ldots,P_{n-1}$ have been built, for some $n\geq 1$. In order to define $P_n$, we distinguish two cases. If $(i_n,p_n)=(i_{n-1},p_{n-1})$ (meaning that $s_n$ and $s_{n-1}$ both belongs to $E_{p_{n-1}}(i_{n-1})=E_{p_n}(i_{n})$), then we set $P_n = 0$. If not, we will see that we can apply  Corollary \ref{cor-incomplete} to get a polynomial $P_n$ that satisfies the following properties:
		\begin{enumerate}[(a)]
			\item \label{item1-log-FHC-proof}$\sqrt{s_n}\leq \text{val}(P_n)\leq \text{deg}(P_n)\leq s_n$;
			\item \label{item2-log-FHC-proof}$\Vert P_n \Vert_{L_{s_n-N_{p_n}(i_n)}^{i_n}}\leq \frac{r_{p_n}}{s_n^2}$;
			\item \label{item3-log-FHC-proof}$\left\Vert P_n - \left(\vp_{p_n} - \sum_{j=0}^{n-1}P_j\right)\right\Vert _{K_{i_n}} \leq \frac{r_{p_n}}{s_n^2}$.
		\end{enumerate}
		Indeed, writing $s_n=s_n-N_{p_n}(i_n) + N_{p_n}(i_n)$ with $N_{p_n}(i_n)\geq \max\{p_n,N_{i_n,0}\}$ and in view of \eqref{eq-assumpt-log-1},Corollary \ref{cor-incomplete} gives us a polynomial $P_n$ satisfying \eqref{item1-log-FHC-proof} such that the quantities
		\[
		\Vert P_n \Vert_{L_{s_n-N_{p_n}(i_n)}^{i_n}}\quad \text{and}\quad \left\Vert P_n - \left(\vp_{p_n} - \sum_{j=0}^{n-1}P_j\right)\right\Vert _{K_{i_n}}
		\]
		are both less than
		\[
		\left\Vert \vp_{p_n} - \sum_{j=0}^{n-1}P_j\right\Vert _{\overline{U_{i_n}}}\theta_{i_n,s_n-N_{p_n}(i_n)}^{s_n}.
		\]
		Let us estimate $\left\Vert \vp_{p_n} - \sum_{j=0}^{n-1}P_j\right\Vert _{\overline{U_{i_n}}}$. Observe that
		\begin{eqnarray*}
			\left\Vert \vp_{p_n} - \sum_{j=0}^{n-1}P_j\right\Vert _{\overline{U_{i_n}}} & \leq & \Vert \vp_{p_n}\Vert _{D(0,3i_n)} + \left\Vert \sum_{j=0}^{n-1}P_j\right\Vert _{\overline{U_{i_n}}}\\
			& \leq & N_{p_n}(i_n) + \left\Vert \sum_{j=0}^{n-1}P_j\right\Vert _{\overline{D(0,\frac{1}{2})}}G_{i_n}^{s_{n-1}}\\
			& \leq & N_{p_n}(i_n) +nG_{i_n}^{\sqrt{s_n}}\\
			& \leq & 2s_nG_{i_n}^{\sqrt{s_n}},
		\end{eqnarray*}
		where we have used Bernstein's lemma for the second inequality and, for the third one, the fact that $\Vert P_j \Vert _{\overline{D(0,1/2)}} \leq \Vert P_j \Vert _{L^{i_{n-1}}_{s_{n-1}-N_{p_{n-1}}(i_{n-1})}}$ and that $s_{n-1}\leq \sqrt{s_n}$. By \eqref{eq-assumpt-log-1}, we get
		\[
		\left\Vert \vp_{p_n} - \sum_{j=0}^{n-1}P_j\right\Vert _{\overline{U_{i_n}}}\theta_{i_n,s_n-N_{p_n}(i_n)}^{s_n} \leq 2s_nG_{i_n}^{\sqrt{s_n}}\theta_{i_n,s_n-N_{p_n}(i_n)}^{s_n}\leq \frac{r_{p_n}}{s_n^2},
		\]
		which gives \eqref{item2-log-FHC-proof} and \eqref{item3-log-FHC-proof}.
		
		Let us set $f=\sum_n P_n$. Observe that any compact set $L\subset \D$ is contained in all the sets $L_{s_n-N_{p_n}(i_n)}^{i_n}$, except at most finitely many of them. Now, by \eqref{item2-log-FHC-proof}, we have for any $N\in \N$,
		\[
		\sum_{n\geq N}\Vert P_n \Vert _{L_{s_n-N_{p_n}(i_n)}^{i_n}}\leq \sum _{n\geq N}\frac{r_{p_n}}{s_n^2}<\infty.
		\]
		This proves that $f\in H(\D)$. Let us now check that $f$ belongs to $FU_{\log}(\D,K_i)$ for any $i\in \N$, which will be enough to conclude, by definition of $(K_i)_i$. Let us fix $i\in \N$, $p\in \N$ and $l\in E_p(i)$. Let $m:=\min\{n\in \N:\,\{s_{n},s_{n+1},\ldots,l\}\subset E_p(i)\}$. Now, the $E_p(i)$ have been chosen in such a way that, for any $u,v\in \N$ and any $(p,i)\neq (q,j)$ such that $s_v\in E_q(i)$, $s_u\in E_p(i)$ and $s_{v}>s_{u}$, one has $s_{v}>(s_{u})^2$. Therefore, we deduce from items \eqref{item1-log-FHC-proof} and \eqref{item3-log-FHC-proof} that
		\[
		\Vert S_l(f) - \vp_p \Vert _{K_i} = \left\Vert \sum_{j=0}^{m}P_j - \vp _p\right\Vert _{K_i} \leq \frac{r_p}{s_m^2}\leq r_p,
		\]
		This completes the proof.
	\end{proof}
	
	\begin{remark}{\rm (1) Let us observe that this result is optimal in the scale of weighted densities associated to weights of the form $(k^r)_k$ with $r\geq -1$, in the sense that
			\[
			\bigcap_{K\in \MM}FU_{\alpha}(\D,K)=\emptyset,
			\]
			whenever $(k^r)_k \lesssim \alpha$ for some $r>-1$. The reason is that, if $\alpha =(k^r)_k$ for some $r>-1$, it turns out that for any set $E\subset \N$, $\dlow_{\alpha}(E)>0$ is equivalent to $\dlow(E)>0$. We refer to the introduction for the case $r\geq 0$. For the case $ r\in (-1,0)$, we first observe that if $(n_k)_k$ is an increasing enumeration of $E$, then $\dlow_{\alpha}(E)>0$ implies $\vp_{\alpha}(n_k)\leq M \vp_{\alpha}(k)$ for some $M\geq 1$. Hence, since $\vp_{\alpha}(n)$ is equivalent to $\frac{n^{r+1}}{r+1}$, we deduce that $n_k \leq M' k$ for any $k\in\N$ and some $M'\geq 1$ (see Lemma 2.10 in \cite{ErnstMouze2019}).
			
			\medskip	
			
			(2) As in Remark \ref{remark-papad}, notice that the previous proof gives us a bit more than the conclusion of Theorem \ref{thm-log-density-FUTS}. In fact it tells us that there exists $f\in H(\D)$ such that, for any $K\in\MM$ and any $p\in \N$, there exists an increasing sequence $(n_k)_k$, with positive logarithmic density, such that $S_{n_k}(f) \to \vp_p$ uniformly on $K$.
		}
	\end{remark}
	
	\medskip
	
	To conclude this section, let us observe that combining Corollary \ref{cor-incomplete-1} with the proof of Theorem \ref{thm-log-density-FUTS}, we can obtain a slight improvement of Theorem \ref{FHC-Taylor-series}. For $1<r\leq R<\infty$ and $\tau>1$, let us denote by $\MM(\tau,r,R)$ the set consisting of all compact sets $K\in \MM$, contained in the annulus $\{z\in \C:\, r\leq |z|\leq R\}$, with the property that for any  open set $U$ containing $K$,  there exist $\theta\in (0,1)$ and $N\in \N$ such that, for any polynomial $\varphi$ and any $n\geq N$, there exists a polynomial $P_n$ of the form $\sum_{k=n/\tau}^n a_kz^k$, such that
	\[
	\Vert P_n \Vert _{\D} \leq \Vert \varphi \Vert _{\overline{U}}\theta ^n\quad \text{and}\quad \Vert P_n - \varphi \Vert _K \leq \Vert \varphi \Vert _{\overline{U}}\theta ^n.
	\]
	It is clear that for any $1<r\leq R<\infty$ and any $\tau\geq 1$, the set $\MM(\tau,r,R)$ is non-empty. Moreover, it is not difficult to see that  if $K\in \MM(\tau,r,R)$, then  $ \MM(\tau,r,R)$ contains all images of $K$ by rotations centered at $0$. A combination of the above theorems allows us to obtain the following result, whose detailed proof is left to the reader.
	
	\begin{theorem}\label{thm-3-FUTS}For any $\tau \geq 2$ and any $1<r\leq R<\infty$, the set $\bigcap _{K\in \MM(\tau,r,R)}FU(\D,K)$ is dense in $H(\D)$.
	\end{theorem}
	
	\subsection{Open problems}
	
	Several problems arise from the results discussed in the previous subsection.
	\subsubsection{$\alpha$-frequently universal Taylor series for \textit{large} $\alpha$}
	Looking at Proposition \ref{thm-non-FUTS-alpha-fast}, it is natural to wonder whether there exist compact sets supporting $\alpha$-frequently universal Taylor series for weights $\alpha$ such that $\vp_{\alpha}$ does not satisfy the $\Delta_2$ condition. Examples of such $\alpha$ are given by weights in the classes $\EE_{\varepsilon}$, $\varepsilon\in [0,1]$, $\DD_s$, $s\in \N\cup \{\infty\}$ and $\LL_l$, $l\geq 1$, see Example \ref{examples-densities}.

	\medskip
	
	In view of Theorem E and of the general strategy described to construct frequenty universal Taylor series, it is tempting to make the following conjecture.
	\begin{conjecture}For any polar compact set $K\subset \C \setminus \overline{\D}$, there are weights $\alpha$, that do not satisfy the $\Delta_2$ condition, for which the set $FU_{\alpha}(\D,K)$ is not empty.
	\end{conjecture}
	Proving this conjecture would require two tasks: firstly, to obtain a \textit{simultaneous approximation} version of Theorem \hyperref[thme]{E}, that is where the polynomials $P_n$ can be chosen arbitrarily small on $\overline{\D}$; secondly, to build disjoint sets $E_p$, $p\in \N$, with $\dlow_{\alpha}(E_p)>0$ and such that, for any $(n,m)\in E_p\times E_q$ with $p\neq q$ and $n>m$, one has $n > \tau_n m $. It is not clear to us how to attack the first task, since $\overline{\D}\cup K$ is not polar, whatever the compact set $K$. The second task seems to us more reachable. A strategy could consist in adapting the construction of sets $E_p$ done in \cite{ErnstMouze2019,ErnstMouze2021} for very small weighted densities, by taking into account the above separation condition. Subsequently, this could shed some light on the following intuition: while the degree of incompleteness of the polynomial approximants is related to the geometry of the compact set $K$, the degree of incompleteness of the approximants is also linked to the extent of pairwise separations of the sets $E_n$, which itself depends on the weight of the weighted density.
	
	\subsubsection{Frequently universal Taylor series on more compact sets}A second problem comes from the fact that Theorem \ref{FHC-Taylor-series} is stated only for compact sets contained in $\C\setminus \overline{\D}$. This leads to the following:
	\begin{question}\label{question-FUTS-partialD}Do we have $FU(D,K)\neq \emptyset$ for any compact set $K\subset \C\setminus \D$, with connected complement?
	\end{question}
	The difficulty comes from the fact that, in order to build a function in $FU(\D,K)$ for a compact set $K$ with connected complement and such that $K\cap \T\neq \emptyset$, we obviously cannot use Corollary \ref{cor-incomplete-1} for $K$ and $L=\overline{\D}$, as in the proof of Theorem \ref{FHC-Taylor-series}. Rather, an idea would consist in applying inductively this latter to $K$ and $L_n$, where $(L_n)_n$ is an increasing exhaustion of $\D$, as we do in Theorem \ref{thm-log-density-FUTS}. But, doing this way, it is not clear that we can choose a constant $\tau>1$ uniformly with respect to $n$.

	\medskip
	
	Moreover, in view of Theorem \ref{thm-3-FUTS}, it is natural to wonder if it is really crucial to have a uniform control on $\tau$, as in the definition of the set $\MM(\tau)$. The proof of $\bigcap_{K\in \MM}FU(\D,K)=\emptyset$ is based in a crucial way on the fact that approximating anything by the partial sums of some $f\in H(\D)$ on compact sets with diameters that increase to $\infty$ imposes that \textit{too many} (in the sense of the upper density) partial sums of $f$ behave the same on any given compact subset of $\C$. Hence one could think that if we consider approximation on compact sets in $\MM$ that are uniformly bounded, the latter phenomenon may not occur. More precisely, let $1\leq r <R < \infty$ and let us denote by $\MM(r,R)$ the set of all polynomials in $\MM$ that are contained in the annulus $A(r,R):=\{z\in \C:\, r\leq |z|\leq R\}$. Is it true that $\bigcap _{K\in \MM(r, R)}FU(\D,K)\neq \emptyset$? Although we do not have an answer, we believe that this is not. Indeed, it seems to us reasonable to think that, in order to approximate anything in $A(K)$ for compact sets $K$ that are closer and closer to a full annulus $A(r,R)$, it is necessary to have at our disposal polynomials that are less and less incomplete. Then, building frequently universal Taylor series would require more and more spacing between the disjoint sets $E_p$, preventing them from having positive natural lower density. All in all, we make the following conjecture.
	\begin{conjecture}For any $1\leq r< R <\infty$, the set $\bigcap_{K\in \MM(r,R)}FU(\D,K)$ is empty.
	\end{conjecture}
	If this conjecture was valid, it would be a strong improvement of the classical result $\bigcap_{K\in \MM}FU(\D,K)=\emptyset$.

	\bibliographystyle{amsplain}
	\bibliography{refs}
	
\end{document}